\documentclass[12pt]{amsart}
\usepackage{microtype}
\usepackage{amsmath,amssymb,amsthm,mathtools, comment}
\usepackage{amsfonts}
\usepackage[T1]{fontenc}
\usepackage{lmodern,dutchcal}
\usepackage[a4paper,margin=28mm]{geometry}
\usepackage[hidelinks]{hyperref}
\usepackage{youngtab}
\usepackage{tikz}
\numberwithin{equation}{section}
\newtheorem{theorem}{Theorem}[section]
\newtheorem{proposition}[theorem]{Proposition}

\theoremstyle{definition}
\newtheorem{definition}[theorem]{Definition}
\newtheorem{example}[theorem]{Example}
\theoremstyle{remark}
\newtheorem{remark}[theorem]{Remark}
\newcommand{\cH}{\mathcal H}
\newcommand{\hH}{\widehat{\mathcal H}}
\newcommand{\gelement}{\mathbf{h}}

\newcommand{\dualg}{\mathbf{m}}
\newcommand{\mone}{{\mathbf{e}}}
\newcommand{\monebarinv}{\mathbf{f}}

\newcommand{\Ue}{{\mathsf{U}}}
\newcommand{\Sp}{\mathsf S}
\newcommand{\PC}{\mathrm{PC}}
\newcommand{\south}{\mathrm{south}}
\newcommand{\bbeta}{\boldsymbol{\beta}}

\newcommand{\NSym}{\mathrm{NSym}}
\newcommand{\QSym}{\mathrm{QSym}}
\newcommand{\Sym}{\mathrm{Sym}}
\newcommand{\NCSym}{\mathrm{NCSym}}
\newcommand{\NCQSym}{\mathrm{NCQSym}}
\newcommand{\hNCSym}{\mathsf{h}}
\newcommand{\eNCSym}{\mathsf e}
\newcommand{\FQSym}{\mathrm{FQSym}}
\newcommand{\gr}{\mathrm{gr}}
\newcommand{\eps}{\varepsilon}
\newcommand{\id}{\mathrm{id}}
\newcommand{\RHA}{\mathsf{RHA}}
\newcommand{\CHA}{\mathsf{CHA}}
\DeclareMathOperator{\Alg}{Alg}
\DeclareMathOperator{\sgn}{sgn}
\DeclareMathOperator{\st}{st}

\usepackage{xcolor}

\title[A universal framework for Hopf algebras in combinatorics]{A universal framework for Hopf algebras in combinatorics}
\author{Farid Aliniaeifard}
\address{
	F. Aliniaeifard, Research Center for Mathematics and Interdisciplinary Sciences,
	Shandong University \& Frontiers Science Center for Nonlinear Expectations, Ministry of Education,
	Qingdao, Shandong 266237, P.R.~China}
\email{farid@sdu.edu.cn}
\thanks{Farid  Aliniaeifard was  supported in part by the Provincial Natural Science Foundation of Shandong, Project No. ZR2026MS0013.}
\author{Stephanie van Willigenburg}
\address{S. van Willigenburg, Department of Mathematics, University of British Columbia, Vancouver BC V6T 1Z2,
Canada}
\thanks{{Stephanie van Willigenburg was supported in part by the Natural Sciences and Engineering Research Council of Canada.}}
\email{steph@math.ubc.ca}

\date{}
\begin{document}
\begin{abstract}
 The ever-growing plethora of combinatorial objects stemming from Bernstein and raising operators strongly indicates the existence of a universal framework for Hopf algebras in combinatorics,  which will provide a unified understanding of combinatorial objects. We establish such a framework, {by dualizing} the theory of combinatorial Hopf algebras developed by Aguiar, Bergeron, and Sottile, obtaining the theory of rooted Hopf algebras. An object in this category is a graded connected Hopf algebra equipped with a group-like element, called a root, in its completion. These roots dualize the character theory of combinatorial Hopf algebras and provide a natural way to extend Bernstein and raising operators to other rooted Hopf algebras. The resulting constructions extend Schur, Hall--Littlewood, and Macdonald functions to rooted Hopf algebras. As an application, we develop these constructions for the Hopf algebra of symmetric functions in noncommuting variables and the Malvenuto--Reutenauer Hopf algebra of permutations.
\end{abstract}
\maketitle

\tableofcontents

\section{Introduction}\label{sec:intro}
{Numerous functions across mathematics admit tantalizingly similar definitions, indicating an underlying unifying structure yet to be revealed, and whose discovery and study would impact research on these original functions immediately. One such function family are the ubiquitous Schur functions, which are} central objects in many branches of mathematics, including representation theory, algebraic combinatorics, and algebraic geometry, and also arise in  {branches of physics including quantum mechanics and string theory.} Their generalizations include Hall--Littlewood and Macdonald polynomials,  {which have been central to many famed problems including the $n!$ problem solved by Haiman~\cite{Haiman}, and the Shuffle Theorem proved by Carlsson and Mellit~\cite{CM},} as well as analogues in other Hopf algebras such as the immaculate basis of noncommutative symmetric functions~\cite{BBS}, a lift of Schur $Q$-functions to the peak algebra~\cite{JL}, and Schur functions in noncommuting variables~\cite{ALvW}. These related families admit constructions using Bernstein and raising operators, {and hence strongly indicate an underlying unifying framework, all results on which would immediately imply results for each special case, and whose study would allow us to translate problems, tools, and results between these many famed, but intricate, functions. We present such a framework, and recover these functions, and more, as follows.} We introduce the category of rooted Hopf algebras, dual to the category of combinatorial Hopf algebras of Aguiar, Bergeron, and Sottile~\cite{ABS}, and define universal expressions using   a collection of scalar series $\mathbf{C}$ and  a group-like element $\gelement$. Suitable choices of $\mathbf{C}$ and $\gelement$ recover the functions and operators mentioned above.

Operator constructions of symmetric functions provide important foundations for this framework. Bernstein operators for Schur functions first appear in ~\cite{Zel} and the
classical raising-operator formulas are described in
Macdonald's book~\cite[Chapters I and III]{Mac}.
Jing~\cite{Jing} constructed Hall--Littlewood functions using vertex
operators, and Shimozono and Zabrocki~\cite{SZ} generalized these operators
to study generalized Kostka polynomials. For Macdonald polynomials,
Kirillov and Noumi~\cite{KN} developed $q$-difference raising operators,
while Lapointe and Vinet~\cite{LV} obtained Rodrigues formulas using
creation operators. In \cite{BHMPS} Blasiak, Haiman, Morse, Pun, and Seelinger obtained a raising operator formula for Macdonald polynomials.

Let $\cH$ be a graded connected Hopf algebra, of finite dimension in each degree, and $\hH$  be its completion, that is,
\[
\cH=\bigoplus_{n\geq 0}\cH_n \qquad \text{and} \qquad \hH=\prod_{n\geq 0}\cH_n.
\] 
 Any element $a\in \hH$ is a formal sum $a=\sum_{n\geq 0} a_{n}$ with $a_{n}\in \cH_{n}$, and we let $\overline{a}=\sum_{n\geq 0}(-1)^n a_{n}$. A group-like element in the completion, or a root of $\cH$, is 
$\gelement=\sum_{n\geq0}\gelement_n$ where
\[
 \gelement_0=1,\qquad
 \Delta(\gelement_n)=\sum_{i+j=n}\gelement_i\otimes\gelement_j.
\]
The roots of $\cH$ form a group with inverse $S(\gelement)=\sum_{n\geq 0} S(\gelement_n)$ where $S$ is the antipode of $\cH$.  
We call $(\cH,\gelement)$ a  {rooted Hopf algebra}.  Graded duality
identifies its root with a character of $\cH^{\gr}=\bigoplus_{n\geq 0}\cH_n^*$, making the category of rooted Hopf algebras dual to the category of combinatorial Hopf algebras of Aguiar, Bergeron, and
Sottile~\cite{ABS}. Thus the initial object in the category of rooted Hopf algebras is  $(\NSym, H)$ with the unique graded Hopf morphism
\[
 J:\NSym\longrightarrow\cH,\qquad J(H_n)=\gelement_n.
\]
Let $\alpha=(\alpha_1,\dots, \alpha_r)$ be a composition, and for a root $\gelement$, let $\gelement_\alpha=\gelement_{\alpha_1}{\cdots} \gelement_{\alpha_r}$.  
A root $\gelement$ in a {noncommutative} Hopf algebra (resp. commutative Hopf algebra) is \emph{free} if the set
$\{\gelement_\alpha:\alpha\text{ a composition}\}$ (resp. $\{\gelement_\lambda:\lambda\text{ a partition}\}$) is linearly independent; in this case, this set is a basis for $A_\gelement$ defined in Equation~\eqref{eq:A_h}.
 The set  $\{\dualg_\alpha\}$ is the dual of $\{\gelement_\alpha\}$, and {determines} 
\[
 \mone=\sum_{i\geq0}\dualg_{{(1^i)}},\qquad
 \monebarinv=\overline{\mone}^{\,-1}.
\] 
The Bernstein operators on $A_\gelement$ are
\[
 B_{(n,\gelement)}
   =\sum_{i\geq0}(-1)^iL_{\gelement_{n+i}}\mone_i^\perp,
\]
where $L_{\gelement_{n+i}}$ is left multiplication by $\gelement_{n+i}$. More generally, let 
$B_n{(a,\gelement)}
   =\sum_{i\geq0}a^iL_{\gelement_{n+i}}\mone_i^\perp$.
We also define the two-parameter Bernstein operators
\[
 \widetilde B_n(a,b;\gelement)
   =\sum_{i\geq0}b^iB_{n+i}(a,\gelement)\monebarinv_i^\perp.
\]
Then we show that
$$J\widetilde B_{n}(a,b;H)=\widetilde{B}_{n}(a,b;\gelement)J,$$
which shows that these generalizations are natural lifts of Bernstein operators of $\NSym$. At $(a,b)=(-1,0)$ and $(-1,q)$, their iterations in $(\NSym, H)$ give the immaculate functions and the modified Hall--Littlewood functions of  
   of Berg, Bergeron, Saliola, Serrano, and
Zabrocki~\cite{BBS}. 

Let $K$ be a commutative $\mathbb C$-algebra.  Fix an integer $r\geq0$ and a matrix
$\mathbf C=(C_{ij})_{1\leq i<j\leq r}$, where
\[
 C_{ij}(v)=\sum_{k\geq0}c_{ij,k}v^k \in K[[v]],\qquad c_{ij,0}=1.
\]
For an integer tuple $\alpha$ of length $r$,  we define our universal element, 
\begin{equation*}\label{eq:introW}
\Ue_\alpha(\mathbf{C};\gelement)
 =\left(\prod_{1\leq i<j\leq r}C_{ij}(R_{ij})\right)\gelement_\alpha,
 \qquad R_{ij}\alpha=\alpha+\epsilon_i-\epsilon_j.
\end{equation*}
Here $\epsilon_i$ is the $i$th coordinate vector. All shifts act on the integer
tuple before the ordered product is evaluated, with
$\gelement_0=1$ and $\gelement_d=0$ for $d<0$.
When all $C_{ij}$ are equal to one of the series 
\[
 C(v)=1-v,\qquad C(v)=\frac{1+av}{1-bv}
\]
$\Ue_\alpha(\mathbf{C};\gelement)$ recovers the value of iterated Bernstein and two-parameter Bernstein operators at $1$, respectively. The specialization $(a,b)=(-1,q)$  recovers modified Hall--Littlewood polynomials.

We can also  recover the
Macdonald raising formula of Blasiak, Haiman, Morse, Pun, and
Seelinger~\cite{BHMPS}. Their column-diagram data specify a matrix
$\mathbf C^\beta(q,t)$. Evaluating it at the inverse signed root
$\mathbf{g}=\overline{\gelement}^{-1}$ gives 
\[
 \widetilde{\mathsf M}_{\mu}(\gelement;q,t)
 =\Ue_{(1^{|\mu|})}({\mathbf C^{\mu^{\mathsf t}}(q,t)};\mathbf{g}).
\]
In $(\Sym,h)$ with $h=\sum_{n\geq0}h_n$, where $h_n$ is the complete homogeneous symmetric function, the inverse signed root is $e=\sum_{n\geq0}e_n$, where $e_n$ is the elementary symmetric function and this lift is the modified
Macdonald polynomial $\widetilde H_\mu(X;q,t)$. 

To extend these constructions from the root-generated subalgebra to the rooted Hopf algebra, we need to extend the action of the raising operators in a natural way,  so that we generalize the Schur, Hall--Littlewood, and Macdonald polynomials. This can be done for several well-known Hopf algebras  such as symmetric functions in noncommuting variables $\NCSym$, the Malvenuto--Reutenauer Hopf algebra of
permutations $\FQSym$, and the quasisymmetric functions in noncommuting variables $\NCQSym$. As demonstrations of our method, we present the cases of $\NCSym$ and $\FQSym$ by constructing raising operators on set partitions and permutations.

In Section~\ref{sec:completion}, we show that the group of characters of a combinatorial Hopf algebra is isomorphic to a group of roots. In Section~\ref{sec:rooted}, we establish a duality between the categories of rooted Hopf algebras and combinatorial Hopf algebras. Section~\ref{sec:operators} extends Bernstein operators to rooted Hopf algebras, while Section~\ref{subsec:seriess} introduces our universal element and shows how to obtain Schur, modified Hall--Littlewood, and Macdonald polynomials for any rooted Hopf algebra. In Sections~\ref{sec:NCSym} and \ref{sec:FQSym}, we extend the constructions to the Hopf algebra of symmetric functions in noncommuting variables and the Malvenuto--Reutenauer Hopf algebra of permutations.

\section{Completions and group-like elements}\label{sec:completion}
Throughout this paper, $\cH=\bigoplus_{n\geq0}\cH_n$ is a graded connected Hopf algebra
over $\mathbb C$, with $\dim\cH_n<\infty$ for every $n$. Its antipode and
counit are denoted by $S$ and $\eps$. 

\subsection{The completion and its duality}
The {\emph{full dual, graded dual,}} and {\emph{completion}} of $\cH$ are
$$
 \cH^*=\prod_{n\geq0}\cH_n^*,\qquad
 \cH^{\gr}=\bigoplus_{n\geq0}\cH_n^*, \text{~ and}\qquad
 \hH=\prod_{n\geq0}\cH_n,
$$
respectively.
Since $\dim\cH_n<\infty$ for all $n$,
$\hH=(\cH^{\gr})^*$. 

Any element $a\in \hH$ is a formal sum $a=\sum_{n\geq 0} a_{n}$ with $a_{n}\in \cH_{n}$. Let $u$ be an indeterminate. The \emph{associated
series}, \emph{signed element}, and \emph{signed series} of $a$ are
$$
 a(u)=\sum_{n\geq0}a_nu^n,\qquad 
 \overline{a}=\sum_{n\geq0}(-1)^na_n,\text{~and} \qquad \overline{a}(u)=a(-u),
$$
respectively. We let
$$
 ab=\sum_{n\geq0}\sum_{i+j=n}a_ib_j, \text{~and}
 \qquad S(a)=\sum_{n\geq0}S(a_n).
$$
We also define $$\Delta(a)=\sum_{n\geq 0} \Delta(a_{n}).$$ 
\begin{remark}
The completion $\hH$ is not in general a Hopf algebra with the algebraic
tensor product. Its coproduct takes values in the completed tensor product
$\prod_{i,j\geq0}\cH_i\otimes\cH_j$. 
\end{remark}

\begin{definition}
An element $\gelement=\sum_{n\geq0}\gelement_n\in\hH$ is
\emph{group-like}, {or a \emph{root} of $\cH$,} if $\gelement_0=1$ and
\begin{equation}\label{eq:dividedpowers}
 \Delta(\gelement_n)=\sum_{i+j=n}\gelement_i\otimes\gelement_j
 \qquad(n\geq0).
\end{equation}
We denote the set of all group-like elements by $G(\hH)$. We also say $\gelement(u)$ is \emph{group-like} if  $\gelement$ is group-like. 
\end{definition}

For an algebra $A$, a {\emph{character}} is an algebra morphism from 
$A$ to $\mathbb{C}$. For a Hopf algebra $\cH$, using Sweedler notation, the {\emph{convolution product}} of
$f,g\in\operatorname{Hom}(\cH,\mathbb C)$ is
\[
 (f*g)(a)=\sum_{(a)}f(a_{(1)})g(a_{(2)}).
\]
The characters form a group under convolution. We use the canonical
pairing $\langle a,f\rangle=f(a)$  which maps $A\otimes A^{*}$ to $\mathbb{C}$, with the algebra element in the first
position.

\begin{theorem}\label{thm:duality}
The set $G(\hH)$ is a group, with inverse
$\gelement^{-1}=S(\gelement)$. The map
$$
\begin{array}{ccc}
 G(\hH)&\longrightarrow&\Alg(\cH^{\gr},\mathbb C),\\
  \gelement&\longmapsto& f_{\gelement}
\end{array}
$$ where 
\[
 f_{\gelement}(a)=\sum_{n\geq 0}\langle \gelement_n, a\rangle
\]
is a group isomorphism.
\end{theorem}
\begin{proof}
For two group-like {elements } $\mathbf{g},\gelement$,  {observe that}
$$\Delta((\mathbf{g}\gelement)_n)=\sum_{i+j=n} (\mathbf{g}\gelement)_i \otimes (\mathbf{g}\gelement)_j{,}$$
and  $(\mathbf{g}\gelement)_0=1$. Thus, $\mathbf{g}\gelement$ is a group-like element. 
For $n>0$, the antipode axioms give
\[
 \sum_{i=0}^n S(\gelement_i)\gelement_{n-i}=0,
 \qquad
 \sum_{i=0}^n \gelement_iS(\gelement_{n-i})=0.
\]
For $n=0$ both sums equal $1$. Thus 
$S(\gelement)\gelement=\gelement S(\gelement)=1$, and so $\gelement^{-1}=S(\gelement)$.
Let  {$\tau$ be defined as} $\tau(a\otimes b)=b\otimes a$.  {Then the} inverse of $\gelement$ is also group-like, since  {the antipode is an antihomomorphism, and so}
{
\begin{align*}
 \Delta((\gelement^{-1})_n)
 &=\Delta(S(\gelement)_n)
   =(S\otimes S)\tau(\Delta(\gelement_n))\\
 &=\sum_{i+j=n}S(\gelement_i)\otimes S(\gelement_j)
   =\sum_{i+j=n}(\gelement^{-1})_i\otimes(\gelement^{-1})_j.
\end{align*}
}

For $a\in\cH_i^*$ and $b\in\cH_j^*$, we have
\begin{align*}
 f_{\gelement}(ab)
 &=\langle\gelement_{i+j},ab\rangle
 =\langle\Delta(\gelement_{i+j}),a\otimes b\rangle\\
 &=\sum_{p+q=i+j}\langle\gelement_p,a\rangle
                         \langle\gelement_q,b\rangle
 =\langle\gelement_i,a\rangle\langle\gelement_j,b\rangle.
\end{align*}
The last equality follows because the graded pairing is zero in unequal
degrees. Moreover, 
$f_{\gelement}(1)=1$, so $f_{\gelement}$ is a character.

Conversely, let $f:\cH^{\gr}\to\mathbb C$ be a character.
Since $\dim \cH_n<\infty$, $(\cH_n^*)^*\cong\cH_n$ and so there is a
unique $\gelement_n$ with
$\langle\gelement_n,a\rangle=f(a)$ for every $a\in\cH_n^*$. Now we show that $\gelement=\sum_{n\geq 0}\gelement_n$ is a group-like element. 
For $a\in\cH_i^*$ and $b\in\cH_j^*$,
$$
 \left\langle\Delta(\gelement_{i+j})
       -\sum_{p+q=i+j}\gelement_p\otimes\gelement_q,
       a\otimes b\right\rangle
 =f(ab)-f(a)f(b)=0.
$$
Since the pairings are nondegenerate, we have that $$\Delta(\gelement_{i+j})
       -\sum_{p+q=i+j}\gelement_p\otimes\gelement_q=0.$$ Also, $f(1)=1$, so $\gelement$ is a group-like element. 

We  show that the map is a group morphism. For
$a\in\cH_n^*$,  we have
\[
\langle \mathbf{g}_i \mathbf{h}_j, a \rangle
= \sum_{(a)} \langle \mathbf{g}_i, a_{(1)} \rangle \langle \mathbf{h}_j, a_{(2)} \rangle.
\]
Therefore,
\[
f_{\mathbf{gh}}(a)
=\sum_{n\geq 0}\sum_{i+j=n} \langle \mathbf{g}_i\gelement_j, a\rangle
=  \sum_{n \geq 0} \sum_{i+j=n} \sum_{(a)} \langle \mathbf{g}_i, a_{(1)} \rangle \langle \mathbf{h}_j, a_{(2)} \rangle
\]
\[
= \sum_{(a)} \left( \sum_{i \geq 0} \langle \mathbf{g}_i, a_{(1)} \rangle \right)
        \left( \sum_{j \geq 0} \langle \mathbf{h}_j, a_{(2)} \rangle \right)=(f_{\mathbf g}*f_{\gelement})(a).
\]
\end{proof}

\subsection{Examples}
In $\Sym$, write $h=\sum_{n\geq0}h_n$ and $e=\sum_{n\geq0}e_n$ for the
complete homogeneous and elementary symmetric function series. Both are group-like, and
$$
 S(h(u))=h(u)^{-1}=e(-u),\qquad
 q(u)=h(-u)^{-1}h(u)=e(u)h(u).
$$
The series $h(u)$, $h(-u)$, $h(u)^{-1}$, and $q(u)$ are group-like, and
$q(-u)=q(u)^{-1}$.

In $\NSym=\mathbb C\langle H_1,H_2,\ldots\rangle${, let $H(u)=\sum_{n\geq0}H_nu^n$ with $H_0=1$, and $E(u)=H(-u)^{-1}$; see~\cite{GKLRT}.} Also
{ 
\begin{gather*}
 \Delta(H_n)=\sum_{i+j=n}H_i\otimes H_j,\qquad S(H_n)=(-1)^nE_n,\\
 Q(u)=H(-u)^{-1}H(u)=E(u)H(u).
\end{gather*}
}


\section{Rooted Hopf algebras}\label{sec:rooted}
\begin{definition}
A \emph{rooted Hopf algebra} is a pair $(\cH,\gelement)$ where $\cH$ is a graded connected Hopf algebra and
$\gelement\in G(\hH)$. A \emph{rooted morphism}
$\Psi:(\cH,\gelement)\to(\cH',\gelement')$ is a graded Hopf morphism
satisfying $\Psi(\gelement_n)=\gelement'_n$ for all $n$. The rooted Hopf algebras with rooted morphisms form the category $\RHA$.
\end{definition}
 {As throughout the paper, all Hopf algebras in these categories have finite-dimensional homogeneous components.}
A \emph{combinatorial Hopf algebra} is a pair $(\cH,\zeta)$, where
$\zeta:\cH\to\mathbb C$ is a character.  {A {\emph{morphism of combinatorial Hopf algebras}} $\Psi$ from $(\cH,\zeta)$ to $(\cH',\zeta')$ is a graded Hopf morphism satisfying}
$\zeta'\circ\Psi=\zeta$. We write $\CHA$ for the category of combinatorial Hopf algebras with its morphisms. 

Let \(\mathsf{C}\) be a category. The \emph{opposite category} (or \emph{dual category}) \(\mathsf{C}^{\mathrm{op}}\) is the category with the same objects as $\mathsf{C}$; however, for every pair of objects $A,B\in\mathsf{C}$
$$
\mathrm{Hom}_{\mathsf{C}^{\mathrm{op}}}(A, B) = \mathrm{Hom}_{\mathsf{C}}(B, A).
$$

Let \(\mathsf{C}\) and \(\mathsf{D}\) be categories. A functor
\[
\mathsf{F} : \mathsf{C} \to \mathsf{D}
\]
is an \emph{equivalence of categories} if there exists a functor
\[
\mathsf{G} : \mathsf{D} \to \mathsf{C}
\]
together with natural isomorphisms
\[
\mathsf{G} \circ \mathsf{F} \cong \mathrm{id}_{\mathsf{C}},
\qquad
\mathsf{F} \circ \mathsf{G} \cong \mathrm{id}_{\mathsf{D}}.
\]
The categories \(\mathsf{C}\) and \(\mathsf{D}\) are then said to be \emph{equivalent}, written
\[
\mathsf{C} \simeq \mathsf{D}.
\]
Also, we say the category $\mathsf{C}$ is the dual of the category $\mathsf{D}$ if $\mathsf{C}^{\mathrm op}\simeq\mathsf D$. 
\begin{theorem}\label{thm:initial}
The graded duality functor that sends $(\cH,\gelement)$ to $(\cH^\gr, f_\gelement)$ gives an equivalence
$\RHA^{\mathrm{op}}\simeq\CHA$. Thus ${\RHA}$ is the dual of $\CHA$. The rooted Hopf algebra $(\NSym,H)$
is the initial object in $\RHA$.  More precisely, for each $(\cH,\gelement)$ there is a unique
rooted morphism
\begin{equation*}\label{eq:J}
 J:(\NSym,H)\longrightarrow(\cH,\gelement),
 \qquad J(H_n)=\gelement_n.
\end{equation*}
The  image of $J$,
\begin{equation}\label{eq:A_h}
A_{\gelement}=\mathbb C\langle\gelement_1,\gelement_2,\ldots\rangle,
\end{equation}
is a cocommutative graded Hopf subalgebra of $\cH$.
\end{theorem}

\begin{proof}
By Theorem~\ref{thm:duality}, the map $\gelement\mapsto f_\gelement$
is a group isomorphism
$G(\hH)\to\Alg(\cH^{\gr},\mathbb C)$.
This yields a functor
\[
 \mathbf F:\RHA^{\mathrm{op}}\to\CHA,\qquad
 (\cH,\gelement)\mapsto(\cH^\gr,f_\gelement),
\]
sending a rooted morphism $\Psi: \cH \longrightarrow \cH'$ to its graded dual $\Psi^\gr: (\cH')^{\gr} \longrightarrow \cH^{\gr}$.
Indeed, for $a\in(\cH'_n)^*$,
\[
 f_{\gelement}(\Psi^{\gr}(a))=\langle \gelement, \Psi^{\gr}(a)\rangle
=\langle\Psi({\gelement}),a\rangle =\langle \gelement', a\rangle
 =f_{\gelement'}(a),
\]
where $\Psi^\gr$ is the dual of $\Psi$. 

Conversely, given $(\cH^{\gr},\zeta)\in\CHA$, by  Theorem~\ref{thm:duality}
we have a unique $\gelement\in G(\hH)$ with $f_\gelement=\zeta$.
Define
\[
 \mathbf G:\CHA\to\RHA^{\mathrm{op}},\qquad
 (\cH^{\gr},\zeta)\mapsto(\cH,\gelement).
\]
For a $\CHA$-morphism $\Psi:(\cH^{\gr},\zeta)\to((\cH')^{\gr},\zeta')$, the
identity $\zeta'\circ\Psi=\zeta$ gives, for $b\in\cH_n^*$,
\[
 \langle\gelement'_n,\Psi(b)\rangle
 =\zeta'(\Psi(b))=\zeta(b)
 =\langle\gelement_n,b\rangle,
\]
hence $\langle\Psi^{\gr}(\gelement'_n),b\rangle=\langle \gelement_n, b\rangle$, and so 
 {$\Psi^{\gr}(\gelement'_n)=\gelement_n$}, thus, $\Psi^{\gr}$ is rooted. 
 {Therefore, canonical graded dual identifications give natural isomorphisms between the two composites and the identity functors.} Thus $\RHA^{\mathrm{op}}\simeq\CHA$.

Since $\NSym$ is freely generated by the $H_n$ with
$\Delta(H_n)=\sum_{i+j=n}H_i\otimes H_j$, the assignment
$J(H_n)=\gelement_n$ extends uniquely to an algebra morphism
$J:\NSym\to\cH$.  It is rooted by construction, and uniqueness
is forced by freeness of $\NSym$. Hence $(\NSym,H)$ is initial in $\RHA$.
 {Its image $A_{\gelement}$ is a graded Hopf subalgebra and is cocommutative because $\NSym$ is cocommutative.}
\end{proof}

\begin{remark}
The map $J$ also appears in the proof of \cite[Theorem 4.1]{ABS} and in \cite[Theorem 2.15]{Zhao}.
\end{remark}
\begin{remark}\label{rem:ABSduality}
The graded duality of Theorem~\ref{thm:initial} also allows further
constructions of Aguiar, Bergeron, and Sottile~\cite{ABS} to be dualized
in the present setting. For example, for any root $\gelement$, we set $I_{\gelement}$ to be the two-sided ideal of $\cH$ defined by
$
 I_{\gelement}=\langle E_n(\gelement):n>0\rangle,
$
where $E_n(\gelement)=\sum_{i=0}^n(-1)^i \gelement_i \gelement_{n-i}.$ Then $I_\gelement$ is a Hopf ideal and $(\cH/I_\gelement)^\gr$ identifies the elements in $\cH^\gr$ whose homogeneous components satisfy the generalized Dehn--Sommerville relations of
\cite[Section 5]{ABS}, and the annihilators of $I_\gelement$ satisfy the Bayer--Billera relations \cite{BB} in the
form of {\cite[Example 5.10]{ABS}}. Moreover, the dual of the descent-to-peak map in {\cite[Example 4.9]{ABS}} is 
    \[
    \begin{array}{cccc}
    \theta:& {\NSym}& \longrightarrow &\NSym\\
     &H_n&\rightarrow& E_n(H).
    \end{array}
    \] 
\end{remark}

A composition $\alpha{\vDash} n$ is a finite sequence of positive integers, {called \emph{parts}},
with sum $n$; its length is {its number of parts,} $\ell(\alpha)$. { A part $a$ repeated $i$ times is often abbreviated to $a^i$.} A partition $\lambda\vdash n$ is a weakly decreasing composition of $n$.  The empty composition has size
and length zero. For a group-like element $\gelement=\sum_{n\geq 0}\gelement_n$, let $\gelement_0=1$, $\gelement_j=0$ for $j<0$. For a composition $\alpha=(\alpha_1,\ldots,\alpha_r)$, let
$\gelement_\alpha=\gelement_{\alpha_1}\cdots\gelement_{\alpha_r}$.

\begin{definition}\label{def:free}
The root $\gelement$ in a {noncommutative} Hopf algebra (resp. commutative Hopf algebra) is \emph{free} if the set
$\{\gelement_\alpha:\alpha\text{ a composition}\}$ (resp. $\{\gelement_\lambda:\lambda\text{ a partition}\}$) is linearly independent; in this case, this set is a basis for $A_\gelement$ defined in Equation~\eqref{eq:A_h}.
\end{definition}

 {For every integer tuple $\gamma=(\gamma_1,\ldots,\gamma_r)$, we also write $\gelement_\gamma=\gelement_{\gamma_1}\cdots\gelement_{\gamma_r}$, with the same conventions $\gelement_0=1$ and $\gelement_d=0$ for $d<0$.} {Let $\gelement_{\alpha - \gamma}=\gelement_{\alpha_1 - \gamma_1}\cdots\gelement_{\alpha_r - \gamma_r}$.}

For a free root $\gelement$, we fix the following notation.
Let $\{\dualg_\beta\}$ be the basis of $A_{\gelement}^{\gr}$ defined by
\begin{equation}\label{eq:dualnotation}
 \langle\gelement_\alpha,\dualg_\beta\rangle=\delta_{\alpha,\beta}
\end{equation}
{where} the indices $\alpha,\beta$ above run over compositions in the noncommutative case and partitions in the commutative case, {and $\delta_{\alpha,\beta}= 1$ if $\alpha = \beta$ and 0 otherwise.}
 Let $\mone_i=\dualg_{{(1^i)}}$ and $\mone=\sum_{i\geq0}\mone_i$; moreover, 
\begin{equation}\label{eq:fnotation}
 \monebarinv=\overline{\mone}^{\,-1}=S(\overline{\mone}),\qquad
 \monebarinv(u)=\mone(-u)^{-1}=\sum_{i\geq0}\monebarinv_i u^i.
\end{equation}

The empty index gives $\mone_0=\monebarinv_0=\eps$.
These functionals are uniquely determined on $A_{\gelement}$.
Both $\mone$ and $\monebarinv$ are group-like in
$\widehat{A_{\gelement}^{\gr}}$. Under $J^{\gr}$ they correspond to
$\sum_i M_{{(1^i)}}$ and $\sum_i F_{(i)}$, respectively. In particular, when $A_{\gelement}^\gr$ is {noncommutative,} we have
\begin{equation}\label{eq:ftom}
    \monebarinv_i=\sum_{\beta{\vDash} i}\dualg_\beta,
\end{equation}
and when it is commutative, 
\begin{equation}\label{eq:ftom1}
    \monebarinv_i=\sum_{\lambda\vdash  i}\dualg_\lambda.
\end{equation}

\subsection{Examples}
In $\Sym$, the root $h$ is free. Use the Hall inner product
$\langle h_\lambda,m_\mu\rangle=\delta_{\lambda,\mu}$ for partitions
$\lambda,\mu$. The corresponding dual series are
\begin{equation*}\label{eq:halldual}
 \dualg_\lambda\longleftrightarrow m_\lambda,\qquad
 \mone_i\longleftrightarrow m_{{(1^i)}}=e_i,
 \qquad \monebarinv_i\longleftrightarrow h_i.
\end{equation*}

In $\NSym$, under the standard pairing
$\langle H_\alpha,M_\beta\rangle=\delta_{\alpha,\beta}$ between
$\NSym$ and $\QSym$, these elements correspond to
\begin{equation*}\label{eq:standarddual}
 \dualg_\alpha\longleftrightarrow M_\alpha,\qquad
 \mone_i\longleftrightarrow M_{{(1^i)}}=e_i,
 \qquad \monebarinv_i\longleftrightarrow h_i=F_{(i)}.
\end{equation*}

\section{Bernstein series and their specializations}\label{sec:operators}
For $f\in\cH^{\gr}$, define
$$
\begin{array}{cccccc}
   f^\perp:  & \cH & \xrightarrow{\Delta}& \cH\otimes \cH& \xrightarrow{f\otimes \id}  \mathbb{C} \otimes \cH \xrightarrow{\cong} \cH.
\end{array}$$  Using Sweedler notation, 
\begin{equation*}\label{eq:skew}
 f^\perp(a)=\sum_{(a)}\langle a_{(1)},f\rangle a_{(2)}.
\end{equation*}
By \cite[Proposition 2.8.2]{GR}, $\langle f^\perp(a),b\rangle=\langle a,fb\rangle$ for
$b\in\cH^{\gr}$, and $(fg)^\perp=g^\perp f^\perp$.
Note that for $a_n\in \cH_n$  {and $g_i\in\cH_i^*$},
\[
 g_i^\perp(a_n)=(g_i\otimes\id)\Delta(a_n)
       \in\cH_{n-i}\quad(0\leq i\leq n),
 \qquad g_i^\perp(a_n)=0\quad(i>n).
\]

{ 
\begin{definition}
Let $g=\sum_{i\geq0}g_i$ and $s=\sum_{j\geq0}s_j$ belong to
$\widehat{\cH^{\gr}}$, with $g_i\in\cH_i^*$ and $s_j\in\cH_j^*$.
For $a=\sum_{n\geq0}a_n\in\hH$, define
\[
 g_i^\perp(a)=\sum_{m\geq0}g_i^\perp(a_{m+i})\in\hH,
 \qquad
 g^\perp(u)(a)=\sum_{i\geq0}g_i^\perp(a)u^i\in\hH[[u]].
\]
These formal series are always defined, as is
\[
 g^\perp(u)s^\perp(v)(a)
 =\sum_{i,j\geq0}g_i^\perp s_j^\perp(a)u^iv^j
 \in\hH[[u,v]].
\]
The specialization $g^\perp(a)=g^\perp(1)(a)$ is defined whenever,
for every $m\geq0$, $g_i^\perp(a_{m+i})=0$ for all sufficiently large $i$;
then
\[
 g^\perp(a)=\sum_{m\geq0}\sum_{i\geq0}g_i^\perp(a_{m+i})\in\hH.
\]
In particular, this condition holds for $a\in\cH$.
\end{definition}
}

 {For a free root, the functionals $\mone_i$ and $\monebarinv_i$ are defined on $A_{\gelement}$, and their skew operators below act on $A_{\gelement}$. Their formal series extend coefficientwise to $\widehat{A_{\gelement}}$. We write $L_c$ for left multiplication by $c$.}

 Equation~\eqref{eq:fnotation} gives
\begin{equation*}\label{eq:skewinverse}
 \mone^\perp(-u)\monebarinv^\perp(u)
 =\monebarinv^\perp(u)\mone^\perp(-u)=\id.
\end{equation*}

{ 
\begin{definition}
For a free root $\gelement$, the \emph{Bernstein--Laurent series} is 
\begin{equation*}\label{eq:Bseries}
 B(a,\gelement;z)
 =L_{\gelement(z)}\mone^\perp(a/z)
   =\sum_{n\in\mathbb Z}B_n(a,\gelement)z^n,
   \end{equation*}
   where
   \begin{equation*}\label{eq:Bmodes}
 B_n(a,\gelement)
=\sum_{i\geq0}a^iL_{\gelement_{n+i}}\mone_i^\perp,
\end{equation*}
and
the \emph{Bernstein operator of index $n$} is
\begin{equation*}\label{eq:Bsigned}
 B_{(n,\gelement)}=B_n(-1,\gelement)
   =\sum_{i\geq0}(-1)^iL_{\gelement_{n+i}}\mone_i^\perp.
\end{equation*}
\end{definition}

\begin{definition}
For a free root $\gelement$, let
\begin{align}
 \widetilde B(a,b;\gelement;z)
   &=L_{\gelement(z)}\mone^\perp(a/z)\monebarinv^{\perp}(b/z)=B(a,\gelement;z)\monebarinv^\perp(b/z),\\
   &=\sum_{n\in \mathbb{Z}}\widetilde{B}_n(a,b;\gelement)z^n,
   \end{align}
   where 
   \begin{align}
 \widetilde B_n(a,b;\gelement)
   &=\sum_{i\geq0}b^iB_{n+i}(a,\gelement)\monebarinv_i^\perp\label{eq:twoparam}
\end{align}
and  the \emph{modified Hall--Littlewood operator}
 is
\begin{equation*}\label{eq:mHL}
     \widetilde B_{(n,\gelement)}
 =\widetilde B_n(-1,q;\gelement)
=\sum_{i\geq0}q^iB_{(n+i,\gelement)}\monebarinv_i^\perp.
\end{equation*}

\end{definition}

}

\subsection{The symmetric and noncommutative symmetric cases}
For $(\Sym,h)$, we have 
\[
 B_{(n,h)}=\sum_{i\geq0}(-1)^iL_{h_{n+i}}e_i^\perp,
 \qquad
 \widetilde B_{(n,h)}
 =\sum_{i\geq0}q^iB_{(n+i,h)}h_i^\perp.
\]
These are the classical Bernstein and modified Hall--Littlewood operators
in the conventions of \cite[Sections~2 and~4]{BBS}.

For $(\NSym,H)$, the monomial quasisymmetric dual basis gives
$\mone_i=M_{{(1^i)}}=F_{{(1^i)}}$ and $\monebarinv_i=F_{(i)}$. Hence
$$
 B_{(n,H)}=\sum_{i\geq0}(-1)^iL_{H_{n+i}}F_{{(1^i)}}^\perp,
 \qquad
 \widetilde{B}_{(n,H)}=\widetilde B_n(-1,q;H)
 =\sum_{i\geq0}q^iB_{(n+i,H)}F_{(i)}^\perp.
$$
These are the immaculate and modified Hall--Littlewood creation operators
of \cite{BBS}.

\subsection{Intertwining and associated Schur functions}\label{sec:schur}
\begin{theorem}\label{thm:intertwine}
For a free root $\gelement$, the universal morphism $J$ satisfies
$$
 J B_{(m,H)}=B_{(m,\gelement)}J,\qquad
 J\widetilde B_m( {a,b};H)=\widetilde B_m( {a,b};\gelement)J.
$$
These are identities on $\NSym$ with values in $A_{\gelement}$.
\end{theorem}
\begin{proof} First we show that 
\[
 J F_{{(1^i)}}^\perp=\mone_i^\perp J,
 \qquad J F_{(i)}^\perp=\monebarinv_i^\perp J.
\]
Fix a composition $\alpha=(\alpha_1,\ldots,\alpha_r)$. We have 
\begin{equation*}
\Delta(\gelement_\alpha)=\Delta(\gelement_{\alpha_1})\cdots \Delta(\gelement_{\alpha_r}), \quad \Delta(\gelement_n)=\sum_{i+j=n}\gelement_{i}\otimes  \gelement_j.
\end{equation*}
Thus 
\begin{equation*}\label{eq:expandedproductcoproduct}
 \Delta(\gelement_\alpha)
 =\sum_{\substack{0\leq t_j\leq\alpha_j\\1\leq j\leq r}}
   (\gelement_{t_1}\cdots\gelement_{t_r})
   \otimes(\gelement_{\alpha_1-t_1}\cdots
                    \gelement_{\alpha_r-t_r}).
\end{equation*}
Note that $\mone_i=\dualg_{{(1^i)}}$, and by Equation \eqref{eq:dualnotation}, we have
\begin{equation*}
 \mone_i^\perp(\gelement_\alpha)= \dualg_{{(1^i)}}^\perp(\gelement_\alpha)=\sum_{\substack{0\leq t_j\leq\alpha_j\\1\leq j\leq r}}
   \langle \gelement_{t_1}\cdots\gelement_{t_r},\dualg_{{(1^i)}}\rangle 
   (\gelement_{\alpha_1-t_1}\cdots
                    \gelement_{\alpha_r-t_r})=\sum_{\substack{\epsilon\in\{0,1\}^r\\|\epsilon|=i}}
                 \gelement_{\alpha-\epsilon}.
 \end{equation*}
By Equations \eqref{eq:ftom} and \eqref{eq:ftom1}, we have
\begin{equation*}
 \monebarinv_i^\perp(\gelement_\alpha)=\sum_{\substack{0\leq t_j\leq\alpha_j\\1\leq j\leq r}}
   \langle \gelement_{t_1}\cdots\gelement_{t_r},\monebarinv_i \rangle
   (\gelement_{\alpha_1-t_1}\cdots
                    \gelement_{\alpha_r-t_r})=\sum_{\substack{\beta\in\mathbb Z_{\geq0}^r\\|\beta|=i}}
                      \gelement_{\alpha-\beta}.
 \end{equation*}
Therefore, 
\begin{equation}\label{eq:flowering}
\mone_i^\perp(\gelement_\alpha)= \sum_{\substack{\epsilon\in\{0,1\}^r\\|\epsilon|=i}}
                      \gelement_{\alpha-\epsilon}, \quad \quad 
 \monebarinv_i^\perp(\gelement_\alpha)
 =\sum_{\substack{\beta\in\mathbb Z_{\geq0}^r\\|\beta|=i}}
                      \gelement_{\alpha-\beta}.
\end{equation}
Since $J(H_n)=\gelement_n$, comparing {Equations \eqref{eq:flowering} with their specializations} to $\NSym$, we have
\[
 J F_{{(1^i)}}^\perp=\mone_i^\perp J,
 \qquad J F_{(i)}^\perp=\monebarinv_i^\perp J.
\]
Also $J L_{H_k}(H_\alpha)= \gelement_k \gelement_\alpha=L_{\gelement_k}J(H_\alpha)$, and so $J L_{H_k}=L_{\gelement_k}J.$ Therefore, 
\begin{align*}
 JB_{(m,H)}(H_\alpha)
 &=\sum_{i\geq0}(-1)^iJ\bigl(H_{m+i}F_{{(1^i)}}^\perp(H_\alpha)\bigr)\\
 &=\sum_{i\geq0}(-1)^i\gelement_{m+i}
                    \mone_i^\perp(\gelement_\alpha)
 =B_{(m,\gelement)}J(H_\alpha).
\end{align*}

Expanding Equation \eqref{eq:twoparam} gives
$$
 \widetilde B_m({a,b};\gelement)
 =\sum_{i,j\geq0}{a^i b^j}
L_{\gelement_{m+i+j}}\mone_i^\perp\monebarinv_j^\perp.
$$ Thus, 
\begin{align*}
 J\widetilde{B}_{m}({a,b};H)(H_\alpha)
 &=\sum_{i,j\geq0}{a^i b^j}
J(L_{H_{m+i+j}}F_{{(1^i)}}^\perp F_{(j)}^\perp(H_\alpha))\\
 &=\sum_{i,j\geq0}{a^i b^j}
L_{\gelement_{m+i+j}}\mone_i^\perp\monebarinv_j^\perp J(H_\alpha)= \widetilde{B}_m({a,b};\gelement)J(H_\alpha).
\end{align*}
\end{proof}

\begin{definition}
For a free root $\gelement$ and a composition $\alpha=(\alpha_1,\ldots,\alpha_r)$,
 the \emph{Schur function associated with $\gelement$} is
\begin{equation*}\label{eq:Schur}
 S_{(\alpha,\gelement)}
 =B_{(\alpha_1,\gelement)}\circ \cdots \circ B_{(\alpha_r,\gelement)}(1){,}
\end{equation*}
and the  \emph{modified Hall--Littlewood function associated with $\gelement$} is
\begin{equation*}\label{eq:tildeSiteration}
 \widetilde S_{(\alpha,\gelement)}(q)
 =\widetilde B_{\alpha_1}(-1,q;\gelement)\circ \cdots
   \circ \widetilde B_{\alpha_r}(-1,q;\gelement)(1).
\end{equation*}

\end{definition}

For a matrix with noncommuting entries, we use the row-ordered determinant
\[
 \det[a_{ij}]_{i,j=1}^r
 =\sum_{\sigma\in\mathfrak S_r}\sgn(\sigma)
       a_{1,\sigma(1)}\cdots a_{r,\sigma(r)}.
\]

\begin{theorem}\label{thm:JT}
For a free root $\gelement$ and every composition $\alpha=(\alpha_1,\dots,\alpha_r)$,
\begin{equation*}\label{eq:JT}
 S_{(\alpha,\gelement)}
 =\det[\gelement_{\alpha_i-i+j}]_{i,j=1}^r
 =\sum_{\sigma\in\mathfrak S_r}\sgn(\sigma)
       \prod_{i=1}^r\gelement_{\alpha_i+\sigma(i)-i},
\end{equation*}
where the product is in increasing row order.
For a partition $\lambda$, $S_{(\lambda,h)}=s_\lambda$ in $\Sym$.
For a composition $\alpha$, $S_{(\alpha,H)}$ is the immaculate function
of \cite{BBS}.
\end{theorem}
\begin{proof}
Consider that 
$$J(S_{(\alpha,H)})=S_{(\alpha, \gelement)}.$$ Since $S_{(\alpha,H)}$ is the immaculate function of \cite{BBS}, the rest 
 follows from \cite[Theorems~2.3 and~3.27]{BBS}.
\end{proof}

\section{Universal creation operators and Macdonald polynomials}\label{subsec:seriess}
Let $(\cH,\gelement)$ be a rooted Hopf algebra. For each integer tuple $\alpha=(\alpha_1,\dots,\alpha_r)$, let  {$\gelement_\alpha=\gelement_{\alpha_1}\cdots\gelement_{\alpha_r}$} where $\gelement_{0}=1$ and $\gelement_{n}=0$ for $n<0$. 

 For an integer tuple  
$\alpha=(\alpha_1,\ldots,\alpha_r)$,
$$R_{ij}(\alpha)=(\alpha_1,\ldots,\alpha_i+1,\ldots,\alpha_j-1,\ldots,\alpha_r).$$  {For nonnegative integers $\lambda_1,\lambda_2,\dots,\lambda_k$, define}
$$\left(R_{i_1j_1}^{\lambda_1}R_{i_2j_2}^{\lambda_2}\cdots R_{i_kj_k}^{\lambda_k}\right)\gelement_\alpha=\gelement_{R_{i_1j_1}^{\lambda_1}R_{i_2j_2}^{\lambda_2}\cdots R_{i_kj_k}^{\lambda_k}(\alpha)}.$$
 {The shifts act on the integer tuple before the ordered product is
evaluated. We extend scalars to a commutative $\mathbb C$-algebra $K$
containing the coefficients of the raising series; rooted morphisms are
extended $K$-linearly.}

Let $C(v)=\sum_{k\geq0}c_kv^k\in K[[v]]$ with $c_0=1$. In $C(v)$, $c_k$ may be treated
formally, so the coefficients may lie in a polynomial or rational
function ring.   For an integer tuple  
$\alpha=(\alpha_1,\ldots,\alpha_r)$, define
\begin{equation}\label{eq:seriesdefinition}
\Ue_\alpha(C;\gelement)=\left(\prod_{1\leq i<j\leq r}C(R_{ij})\right)\gelement_\alpha,
\end{equation} where  $$R_{ij}(\alpha)=(\alpha_1,\ldots,\alpha_i+1,\ldots,\alpha_j-1,\ldots,\alpha_r).$$ 
By convention, $\Ue_\varnothing(C;\gelement)=1$.  For each
$1\leq i<j\leq r$,  {let} 
\begin{equation*}\label{eq:CRij}
C(R_{ij})=\sum_{k_{ij}\geq 0}c_{k_{ij}}R_{ij}^{k_{ij}}.
\end{equation*}
\begin{theorem}\label{prop:WCalpha}
    Let $\gelement$ be a root and let $\alpha=(\alpha_1,\dots, \alpha_r)$ be an  {integer tuple}. Then
\begin{equation}\label{eq:seriesexpansion}
 \Ue_\alpha(C;\gelement)
 =\sum
       \left(\prod_{i<j}c_{k_{ij}}\right)
       \gelement_{\beta_1}\cdots\gelement_{\beta_r}{,}
\end{equation}
where the sum runs over all pairs $$((k_{ij})_{1\leq i<j\leq r}, (\beta_1,\dots, \beta_r))$$ such that $k_{ij}\geq 0$ for $1\leq i<j\leq r$, $\beta_i\geq 0$ for $1\leq i\leq r$, and 
\[
 \beta_i=\alpha_i+\sum_{j>i}k_{ij}-\sum_{j<i}k_{ji}.
\]
Moreover, if \(\Psi : (\mathcal{H}, \gelement) \to (\mathcal{H}', \gelement')\) is a rooted morphism, then
\[
\Psi(\Ue_\alpha(C;\gelement)) = \Ue_\alpha(C;\gelement').
\]
\end{theorem}

    \begin{proof}
Since the operators $R_{ij}$ commute on tuples, we obtain
\[
 \prod_{i<j}C(R_{ij})
 =\sum_{(k_{ij}\geq 0)_{i<j}}
       \Bigl(\prod_{i<j}c_{k_{ij}}\Bigr)
       \prod_{i<j}R_{ij}^{k_{ij}}.
\]
Applying $\prod_{i<j}R_{ij}^{k_{ij}}$ to the tuple $\alpha$ gives the
tuple $\beta=(\beta_1,\ldots,\beta_r)$ with
\[
\alpha_i=\beta_i-\sum_{j>i}k_{ij}+\sum_{j<i}k_{ji}.
\]
We have that 
\[
\sum_{i=1}^r(i-1)\alpha_i
=\sum_{i=1}^r(i-1)\beta_i
-\sum_{i=1}^r(i-1)\sum_{j>i}k_{ij}
+\sum_{i=1}^r(i-1)\sum_{j<i}k_{ji}.
\]
Since 
\[
\sum_{i=1}^r(i-1)\sum_{j>i}k_{ij}
=\sum_{i<j}(i-1)k_{ij}
\qquad \text{and}\qquad
\sum_{i=1}^r(i-1)\sum_{j<i}k_{ji}
=\sum_{i<j}(j-1)k_{ij},
\]
we get
\[
\sum_{i=1}^r(i-1)\alpha_i
=\sum_{i=1}^r(i-1)\beta_i
-\sum_{i<j}(i-1)k_{ij}
+\sum_{i<j}(j-1)k_{ij}.
\]
The two {latter terms} combine into \(\sum_{i<j}(j-i)k_{ij}\), so
\[
\sum_{i<j}(j-i)k_{ij}
=\sum_{i=1}^r(i-1)\alpha_i-\sum_{i=1}^r(i-1)\beta_i.
\]
 {If all $\beta_i\geq0$, then}
{ 
\[
0\leq\sum_{i<j}(j-i)k_{ij}
=\sum_{i=1}^r(i-1)\alpha_i-\sum_{i=1}^r(i-1)\beta_i
\leq\sum_{i=1}^r(i-1)\alpha_i.
\]
}
 {Thus only finitely many arrays $(k_{ij})$ contribute, for every
integer tuple $\alpha$. The coefficients are therefore polynomials in
finitely many coefficients of the raising series.}
Evaluating the ordered product
$\gelement_{\beta_1}\cdots\gelement_{\beta_r}$ yields
\eqref{eq:seriesexpansion}.

For a composition $\alpha$,  {the choice $k_{ij}=0$ for all
$i<j$ in \eqref{eq:seriesexpansion}} gives $\beta=\alpha$ and contributes
$\gelement_\alpha$.  {If some $k_{ij}>0$, choose the least $i$ for
which $k_{ij}>0$ for some $j>i$. Then $\beta_i>\alpha_i$, while
$\beta_1,\ldots,\beta_{i-1}$ remain unchanged.} After deleting zero
parts, the resulting composition is lexicographically larger than
$\alpha$. Hence
\begin{equation*}\label{eq:seriestriangular}
\Ue_\alpha(C;\gelement)
 =\gelement_\alpha+
       \sum_{\beta>_{\mathrm{lex}}\alpha}a_{\alpha\beta}\gelement_\beta.
\end{equation*}

Since $\Psi(\gelement_n)=\gelement'_n$ and $\Psi$ is an algebra morphism, we have that 
\[
\Ue_\alpha(C;\gelement')=\Psi(\Ue_\alpha(C;\gelement)).
\]
\end{proof}

\begin{example}\label{ex:series22}
For two nonnegative parts $a,b$ the only raising operator is $R_{12}$,
so
\[
 \Ue_{(a,b)}(C;\gelement)
   =\sum_{k=0}^b c_k\gelement_{a+k}\gelement_{b-k}.
\]
In particular,
\begin{align*}
 \Ue_{(2,2)}(C;\gelement)
   &=\gelement_2\gelement_2+c_1\gelement_3\gelement_1+c_2\gelement_4,\\
 \Ue_{(2,2)}({1-v};\gelement)
   &=\gelement_2\gelement_2-\gelement_3\gelement_1,\\
 \Ue_{(2,2)}({(1-v)/(1-qv)};\gelement)
   &=\gelement_2\gelement_2+(q-1)\gelement_3\gelement_1
                                      +q(q-1)\gelement_4.
\end{align*}
The last equality uses
$C_q(v)=(1-v)/(1-qv)=1+\sum_{k\geq1}(q-1)q^{k-1}v^k$.
In $\Sym$ these are $s_{{(2,2)}}$ and $Q'_{{(2,2)}}(q)$.
In $\NSym$, they give immaculate and modified Hall--Littlewood functions.
\end{example}

\subsection{From creation operators to the modified Hall--Littlewood raising series}
Recall that $C_q(v)=(1-v)/(1-qv)$ and
$\widetilde B_{(m,\gelement)}=\widetilde B_m(-1,q;\gelement)$.
We show that applying these creation operators successively is the same
as applying the raising operators in the definition of $\Ue_\alpha({C_q};\gelement)$.  A denominator $(1-av)^{-1}$ always means the formal
series $\sum_{k\geq0}a^kv^k$.

\begin{proposition}\label{prop:normal}
For a free root $\gelement$, we have
\begin{equation}\label{eq:eL}
 \mone^\perp(u)L_{\gelement(z)}
   =(1+uz)L_{\gelement(z)}\mone^\perp(u),
   \end{equation} 
   and 
   \begin{equation}\label{eq:fL}
 \monebarinv^\perp(u)L_{\gelement(z)}
   =(1-uz)^{-1}L_{\gelement(z)}\monebarinv^\perp(u).
\end{equation} 
\end{proposition}
\begin{proof}
Since the root is free, the ordered products $\gelement_\alpha$
form a basis of $A_\gelement$, with dual basis $\dualg_\alpha$. Recall that
\[
 \mone_n=\dualg_{{(1^n)}},\qquad
 \monebarinv_n=\sum_{\beta{\vDash} n}\dualg_\beta,
 \qquad \mone_0=\monebarinv_0=\eps.
\]
Thus, for every composition $\alpha$, including the empty one,
\[
 \langle\gelement_\alpha,\mone\rangle
 =\begin{cases}1,&\alpha={(1^{\ell(\alpha)})},\\0,&\text{otherwise},\end{cases}
 \qquad \text{and}\quad
 \langle\gelement_\alpha,\monebarinv\rangle=1.
\]
The product $\gelement_\alpha\gelement_\beta$ is indexed by the
{\emph{concatenation} of $\alpha$ and $\beta$, namely the parts of $\alpha$ followed by the parts of $\beta$.} Hence each of the functionals
$\langle-,\mone\rangle$ and $\langle-,\monebarinv\rangle$ is
multiplicative on basis elements, and therefore on $A_{\gelement}$.

Let $g$ denote either $\mone$ or $\monebarinv$. Taking homogeneous
components of this multiplicativity gives, for $a,b\in A_{\gelement}$,
\[
 \langle ab,g_k\rangle
 =\sum_{i+j=k}\langle a,g_i\rangle\langle b,g_j\rangle.
\]
Since  $\Delta(ab)=\Delta(a)\Delta(b)$, we obtain
\begin{align*}
 g_k^\perp(ab)
 &=\sum_{(a),(b)}\langle a_{(1)}b_{(1)},g_k\rangle
                                      a_{(2)}b_{(2)}\\
 &=\sum_{i+j=k}
   \left(\sum_{(a)}\langle a_{(1)},g_i\rangle a_{(2)}\right)
   \left(\sum_{(b)}\langle b_{(1)},g_j\rangle b_{(2)}\right)\\
 &=\sum_{i+j=k}g_i^\perp(a)g_j^\perp(b).
\end{align*}
For $a,b\in\widehat A_\gelement$, this identity still holds in each degree
$m$. Consequently, as a formal power series in $u$,
\[
 g^\perp(u)(ab)=g^\perp(u)(a)g^\perp(u)(b)
       \qquad(a,b\in\widehat A_{\gelement}).
\]

Next, since
$\Delta(\gelement_n)=\sum_{i=0}^n
\gelement_i\otimes\gelement_{n-i}$, we have 
\[
 g_i^\perp(\gelement_n)
   =\langle\gelement_i,g_i\rangle\gelement_{n-i}
       \quad(0\leq i\leq n),
 \qquad g_i^\perp(\gelement_n)=0\quad(i>n).
\]
For independent indeterminates $u,z$, it follows that
\begin{align*}
 g^\perp(u)(\gelement(z))
 &=\sum_{n\geq0}\sum_{i=0}^n
       \langle\gelement_i,g_i\rangle\gelement_{n-i}u^i z^n\\
 &=\left(\sum_{i\geq0}\langle\gelement_i,g_i\rangle(uz)^i\right)
       \gelement(z).
\end{align*}
Duality gives
\[
 \langle\gelement_i,\mone_i\rangle
 =\begin{cases}1,&i=0,1,\\0,&i\geq2,\end{cases}
 \qquad
 \langle\gelement_i,\monebarinv_i\rangle=1\quad(i\geq0).
\]
Therefore
\[
 \mone^\perp(u)(\gelement(z))=(1+uz)\gelement(z),
 \qquad
 \monebarinv^\perp(u)(\gelement(z))
       =(1-uz)^{-1}\gelement(z),
\]
where $(1-uz)^{-1}=\sum_{i\geq0}(uz)^i$.
\end{proof}

\begin{theorem} For a free root $\gelement$ and every integer tuple
$\alpha=(\alpha_1,\ldots,\alpha_r)$,
 \begin{equation}\label{eq:seriescreation}
 \widetilde B_{\alpha_1}{(a,b;\gelement)}\cdots
 \widetilde B_{\alpha_r}{(a,b;\gelement)}(1)
 =\Ue_\alpha((1+av)/(1-bv);\gelement).
\end{equation}
The empty tuple gives $1$ on both sides.
\end{theorem}

\begin{proof}
By definition,
\[
  \widetilde B{(a,b;\gelement,z)}=\sum_{m\in\mathbb Z}\widetilde B_m{(a,b;\gelement)}z^m
 =L_{\gelement(z)}\mone^\perp(a/z)
                         \monebarinv^\perp(b/z).
\]
Write $D(z)=\mone^\perp(a/z)\monebarinv^\perp(b/z)$
for the two lowering factors. Proposition~\ref{prop:normal} gives
\[
 D(z_i)L_{\gelement(z_j)}
 =\frac{1+az_j/z_i}{1-bz_j/z_i}
       L_{\gelement(z_j)}D(z_i)
 =C(z_j/z_i)L_{\gelement(z_j)}D(z_i),
\]
where $C(z_j/z_i)=({1+az_j/z_i})/({1-bz_j/z_i})$.
Indeed, moving $\monebarinv^\perp(bz_i^{-1})$ past the multiplication
operator contributes $(1-bz_j/z_i)^{-1}$, and moving
$\mone^\perp(a/z_i)$ past it contributes $1+az_j/z_i$.

Note that $D(z_i)(1)=1$, because every positive-degree
lowering operator kills $1$. Thus,
\[
 \left(\sum_{m\in \mathbb{Z}}\widetilde B_m{(a,b;\gelement)}z_1^m\right)\left(\sum_{m\in \mathbb{Z}}\widetilde B_m{(a,b;\gelement)}z_2^m\right)\cdots
 \left(\sum_{m\in \mathbb{Z}}\widetilde B_{m}{(a,b;\gelement)}z_r^m\right)(1)
 \]
 \[
=L_{\gelement(z_1)}D(z_1)\cdots {D(z_{r-2})}L_{\gelement(z_{r-1})}D(z_{r-1}) L_{\gelement(z_r)}D(z_r)(1)
 \]
 \[
 =L_{\gelement(z_1)}D(z_1)\cdots D(z_{r-2})L_{\gelement(z_{r-1})}D(z_{r-1}) L_{\gelement(z_r)}(1)
\]
\[
=L_{\gelement(z_1)}D(z_1)\cdots D(z_{r-2})L_{\gelement(z_{r-1})}C(z_r/z_{r-1})L_{\gelement(z_{r})}D(z_{r-1})(1)
\]
\[
{=C(z_r/z_{r-1})L_{\gelement(z_1)}D(z_1)\cdots D(z_{r-2})L_{\gelement(z_{r-1})}L_{\gelement(z_{r})}D(z_{r-1})(1)}
\]
\[
{=C(z_r/z_{r-1})C(z_{r}/z_{r-2})C(z_{r-1}/z_{r-2})L_{\gelement(z_1)}D(z_1)\cdots L_{\gelement(z_{r-1})}L_{\gelement(z_{r})}D(z_{r-2})D(z_{r-1})(1).}
\]
{We keep} moving each $D(z_i)$ to the right past the later  {left}
multiplication operators. For each pair $i<j$ this contributes one
factor $C(z_j/z_i)$.   {We continue until we obtain}
\[
{\prod_{i<j}C(z_j/z_i) L_{\gelement(z_1)}\cdots L_{\gelement(z_r)}D(z_1)\cdots D(z_{r-2})D(z_{r-1})(1).}
\]
 {Since $D(z_i)(1)=1$, this equals $\prod_{i<j}C(z_j/z_i)\gelement(z_1)\cdots\gelement(z_r).$}
Therefore
\[
 \left(\sum_{m\in \mathbb{Z}}\widetilde B_m{(a,b;\gelement)}z_1^m\right) \cdots
 \left(\sum_{m\in \mathbb{Z}}\widetilde B_{m}{(a,b;\gelement)}z_r^m\right)(1)
 =\gelement(z_1)\cdots\gelement(z_r)
                 \prod_{i<j}C(z_j/z_i).
\]

Now take the coefficient of \(z_1^{\alpha_1} z_2^{\alpha_2} \cdots z_r^{\alpha_r}\) on both sides of
\[
\left(\sum_{m\in \mathbb{Z}}\widetilde B_m{(a,b;\gelement)}z_1^m\right) \cdots
 \left(\sum_{m\in \mathbb{Z}}\widetilde B_{m}{(a,b;\gelement)}z_r^m\right)(1)
= \gelement(z_1) \cdots \gelement(z_r) \prod_{1 \le i < j \le r} C(z_j/z_i).
\]
The left-hand side gives the  {element}
\[
\widetilde B{(a,b;\gelement)}\cdots
 \widetilde B{(a,b;\gelement)}{(1).}
\]
We  {now} consider the right-hand side. Note that
\[
[z_i^c z_j^d] \, (z_j/z_i)^k \, \gelement(z_i) \gelement(z_j) = \gelement_{c+k} \gelement_{d-k}.
\]
In other words, choosing the term \((z_j/z_i)^k\) in a raising series \(C(z_j/z_i)\) has exactly the effect of the raising operator
\[
R_{ij}^k : (\alpha_1,\ldots, \alpha_i, \ldots, \alpha_j, \ldots, \alpha_r) \mapsto (\alpha_1,\ldots, \alpha_i + k, \ldots, \alpha_j - k, \ldots,\alpha_r)
\]
on the indices.

So we expand each \(C(z_j/z_i)\) as a power series in \(z_j/z_i\),
\[
C(z_j/z_i) = \sum_{k \ge 0} c_k (z_j/z_i)^k,
\]
and multiply them before evaluating the complete factors. The resulting coefficient is precisely
\[
\left(\prod_{1 \le i < j \le r} C_q(R_{ij})\right) \, \gelement_\alpha = {\Ue}_\alpha({C};\gelement){.}
\]
\end{proof}

At $(a,b)=(-1,0)$, the operators become $B_{(m,\gelement)}$ and $C=1-v$.  {Thus}
 
\[
B_{(\alpha_1,\gelement)}\cdots B_{(\alpha_r,\gelement)}(1)
=\Ue_\alpha({1-v};\gelement),
\]
{which} equals  the associated Schur function $S_{(\alpha,\gelement)}$ when $\alpha$ is a composition.
 At $(a,b)=(-1,q)$, the operators become $\widetilde B_{m}(-1,q;\gelement)$ and $C=(1-v)/(1-qv)$. Thus 
 $$\widetilde B_{\alpha_1}(-1,q;\gelement)\circ \cdots
   \circ \widetilde B_{\alpha_r}(-1,q;\gelement)(1)=\Ue_{\alpha}((1-v)/(1-qv);\gelement)$$
 which equals the  {modified Hall--Littlewood function associated with {$\gelement$,}} $\widetilde S_{(\alpha,\gelement)}(q)$.

\begin{example}\label{ex:learncreation}
For a free root $\gelement$, $B_{(1,\gelement)}(1)=\gelement_1$.
Only the degree-zero and degree-one lowering terms can act
nontrivially on $\gelement_1$, and
$\mone_1^\perp(\gelement_1)=1$. Thus
\[
 B_{(1,\gelement)}B_{(1,\gelement)}(1)
 =\gelement_1^2-\gelement_2
 =\Ue_{(1,1)}({1-v};\gelement).
\]
 {Also}, consider that $\monebarinv_1^\perp(\gelement_1)=1$, so
\begin{align*}
 \widetilde B_1(-1,q;\gelement)(\gelement_1)
 &=B_{(1,\gelement)}(\gelement_1)
                           +qB_{(2,\gelement)}(1)\\
 &=\gelement_1^2+(q-1)\gelement_2.
\end{align*}
\end{example}

\subsection{Raising series attached to pairs of positions}\label{subsec:pairseriess}
Before we give the raising series for Macdonald  {polynomials}, we {prove} a small extension of Equation~\eqref{eq:seriesdefinition}.
Fix an integer $r\geq0$ and a collection
$\mathbf C=(C_{ij})_{1\leq i<j\leq r}$, where
\[
 C_{ij}(v)=\sum_{k\geq0}c_{ij,k}v^k,\qquad c_{ij,0}=1.
\]
The coefficients are scalars; they may depend on formal parameters.
For an integer tuple $\alpha$ of length $r$, let
\begin{equation*}\label{eq:pairseries}
 \Ue_\alpha({\mathbf C};\gelement)
 =\left(\prod_{1\leq i<j\leq r}C_{ij}(R_{ij})\right)\gelement_\alpha.
\end{equation*}

\begin{proposition}
Let $\gelement$ be a root and let $\alpha=(\alpha_1,\dots, \alpha_r)$ be an  {integer tuple}.  
Then
\begin{equation}\label{eq:UniWp}
   \Ue_\alpha({\mathbf C};\gelement)
= \sum \left( \prod_{i<j} c_{ij,k_{ij}} \right) \gelement_{\beta_1} \cdots \gelement_{\beta_r},
\end{equation}
where the sum runs over all pairs $$((k_{ij})_{1\leq i<j\leq r}, (\beta_1,\dots, \beta_r))$$ such that $k_{ij}\geq 0$ for $1\leq i<j\leq r$, $\beta_i\geq 0$ for $1\leq i\leq r$, and 
\[
 \beta_i=\alpha_i+\sum_{j>i}k_{ij}-\sum_{j<i}k_{ji}.
\]
Moreover, if \(\Psi : (\mathcal{H}, \gelement) \to (\mathcal{H}', \gelement')\) is a rooted morphism, then
\[
\Psi\bigl(\Ue_\alpha({\mathbf C};\gelement)\bigr) = \Ue_\alpha({\mathbf C};\gelement').
\]
\end{proposition}
\begin{proof}
 {Expand each $C_{ij}(R_{ij})$ as a formal series. The proof of
Theorem~\ref{prop:WCalpha}, including its finiteness bound, applies
with $c_{k_{ij}}$ replaced by $c_{ij,k_{ij}}$. }
\end{proof}

\begin{example}\label{ex:learntransfers}
 {Let} $\alpha=(1,1,1)$. 
Let
\[
 A=[v]C_{12}(v),\quad A_2=[v^2]C_{12}(v),\quad
 B=[v]C_{13}(v),\quad D=[v]C_{23}(v).
\]
 Choosing powers
$(k_{12},k_{13},k_{23})$ {forces} the tuple {$(\beta_1, \beta_2, \beta_3)$ to be}
\[
 (1+k_{12}+k_{13},\ 1-k_{12}+k_{23},\ 1-k_{13}-k_{23}).
\]
Nonnegativity {of each $\beta_i$} requires $k_{13}+k_{23}\leq1$ and
$k_{12}\leq1+k_{23}$. The complete list {of tuple pairs} is {as follows.}
\[
\begin{array}{c|c|c|c}
(k_{12},k_{13},k_{23})&(\beta_1, \beta_2, \beta_3)&\text{coefficient}&\text{product}\\\hline
(0,0,0)&(1,1,1)&1&\gelement_1^3\\
(1,0,0)&(2,0,1)&A&\gelement_2\gelement_1\\
(0,1,0)&(2,1,0)&B&\gelement_2\gelement_1\\
(1,1,0)&(3,0,0)&AB&\gelement_3\\
(0,0,1)&(1,2,0)&D&\gelement_1\gelement_2\\
(1,0,1)&(2,1,0)&AD&\gelement_2\gelement_1\\
(2,0,1)&(3,0,0)&A_2D&\gelement_3
\end{array}
\]
Several choices can give the same final product, so their
coefficients must be added. We obtain
\begin{equation}\label{eq:learnthreeformula}
\Ue_{(1,1,1)}({\mathbf C};\gelement)
 =\gelement_1^3+D\gelement_1\gelement_2
 +(A+B+AD)\gelement_2\gelement_1+(AB+A_2D)\gelement_3.
\end{equation}
There is no need for any coefficient of $C_{13}$ or $C_{23}$
beyond degree one, even if these raising series are infinite series.
For a uniform raising series $C=1-v$, substituting
$A=B=D=-1$ and $A_2=0$ gives
\[
 \Ue_{(1,1,1)}({1-v};\gelement)
 =\gelement_1^3-\gelement_1\gelement_2
                  -\gelement_2\gelement_1+\gelement_3.
\]
\end{example}

\subsection{The Macdonald diagram and its raising series}\label{subsec:Macdiagram}
Let $\beta=(\beta_1,\ldots,\beta_r)$ be a composition {of $n$}. Its {\emph{column diagram}} is
$$
 \boldsymbol\beta=\{(k,l):1\leq k\leq r,\ 1\leq l\leq\beta_k\}.
$$

Order the boxes from top to bottom, and from left to right within each row, and call them $\boldsymbol{\beta}[1],\dots, \boldsymbol{\beta}[n]$. For boxes above the bottom row, if  $\boldsymbol{\beta}[i]=(k,l)$, let ${\rm south}(\boldsymbol{\beta}[i])=(k,l-1)$. 
 The {\emph{leg}} of $\boldsymbol{\beta}[i]$ is $\ell_i={\rm leg}(\bbeta[i])=\beta_k-l$, which is the number of boxes strictly north of $\boldsymbol{\beta}[i]$. And the {\emph{arm}} of $\boldsymbol{\beta}[i]$ is 
\begin{align}
 a_i=\operatorname{arm}({\bbeta}[i])&=\#\{d<k:l-1\leq\beta_d<\beta_k\}
       +\#\{d>k:l\leq\beta_d\leq\beta_k\}.
       \label{eq:Macarm}
\end{align}
This is the number of boxes strictly east of $\boldsymbol{\beta}[i]$ in columns of height $\leq \beta_k$ or strictly west
of $\south(\bbeta[i])$ in columns of height $<\beta_k$.

 {Define $s:[n-r]\longrightarrow[n]$ by $s(i)=j$ if
$\south(\bbeta[i])=\bbeta[j]$. Thus $s$ is defined precisely for boxes
above the bottom row. Let}
{ 
\begin{align*}
R_+&=\{(i,j):1\leq i<j\leq n\},\\
R_\beta&=\{(i,j)\in R_+:i\leq n-r,\ s(i)\leq j\},\\
\widehat R_\beta&=\{(i,j)\in R_+:i\leq n-r,\ s(i)<j\}.
\end{align*}
}

 For example, let $\beta=(2, 3, 4, 1, 3, 2, 1, 4)$. Then $\south(\bbeta[3])=\bbeta[8]$, thus,  $s(3)=8$ and $s({17})$ is undefined. For $\bbeta[4]$, $\ell_4={\rm leg}(\bbeta[4])=1$ {from the box $\bbeta[1],$ and $a_4=\mathrm{arm}(\bbeta[4])=4$ from the boxes $\bbeta[5], \bbeta[6], \bbeta[7], \bbeta[8]$.}
$$
\begin{tikzpicture}[scale=0.6]
\foreach \y/\x/\z in {
1/1/13,2/1/7,
1/2/14,2/2/8,3/2/3,
1/3/15,2/3/9,3/3/4,4/3/1,
1/4/16,
1/5/17,2/5/10,3/5/5, 
1/6/18,2/6/11,
1/7/19,
1/8/20,2/8/12,3/8/6,4/8/2}{
      \draw (1.5*\x, 1.5*\y) rectangle (1.5*\x+1.5, 1.5*\y+1.5);
      \node[font=\scriptsize] at (1.5*\x+0.75, 1.5*\y+0.75) { {$\bbeta[{\z}]$}};
    }
\end{tikzpicture}
$$

\begin{definition}\label{def:Macseriess}
 {Let $\beta=(\beta_1,\dots,\beta_r){\vDash} n$.} For $1\leq i<j\leq n,$ define
\begin{equation}\label{eq:Macseriess}
 C^\beta_{ij}(v;q,t)=\frac{1-v}{1-qv}
 \begin{cases}
  \displaystyle\frac{1-q^{a_i+1}t^{-\ell_i}v}{1-tv},
       &j=s(i),\\[3mm]
  \displaystyle\frac{1-qtv}{1-tv},&j>s(i), \\[1mm]
  1,&\text{otherwise.}
 \end{cases}
\end{equation}
In the first two cases $s(i)$ is understood to be defined. Let
\begin{equation}\label{eq:Cbqt}
    \mathbf C^\beta(q,t)=(C^\beta_{ij}(v;q,t))_{1\leq i<j\leq n}.
\end{equation}
\end{definition}

\begin{example}\label{ex:Macdiagram21}
For $\beta=(2,1)$, the {column} diagram is {as follows, and each entry represents its box.}
\[
 \begin{array}{cc}
 \bbeta[1]&\\
 \bbeta[2]&\bbeta[3]
 \end{array}
\]
Only $s(1)$ is defined, and $s(1)=2$, $a_1=\ell_1=0$.
Consequently
\begin{align*}
 C_{12}^\beta(v)&=\frac{1-v}{1-tv},&
 C_{13}^\beta(v)&=\frac{(1-v)(1-qtv)}{(1-qv)(1-tv)},&
 C_{23}^\beta(v)&=\frac{1-v}{1-qv}.
\end{align*}
\end{example}

\begin{example}\label{ex:learnrectangle}
Let $\beta=(2,2)$. The {column} diagram and its box data are
\[
 \begin{array}{cc}\bbeta[1]&\bbeta[2]\\ \bbeta[3]&\bbeta[4]\end{array}
 \qquad
 \begin{array}{c|ccc}
 i&s(i)&a_i&\ell_i\\\hline
 1&3&1&0\\2&4&0&0
 \end{array}.
\]
The bottom boxes have no south box, so their arm and leg data
are not needed in the raising series rule. Substituting in
\eqref{eq:Macseriess} yields {the following.}
\[
\begin{array}{c|c|c}
 (i,j)&\text{relation to the south box}&C^\beta_{ij}(v;q,t)\\\hline
 (1,2)&2<s(1)&(1-v)/(1-qv)\\[1mm]
 (1,3)&3=s(1)&(1-v)(1-q^2v)/((1-qv)(1-tv))\\[1mm]
 (1,4)&4>s(1)&(1-v)(1-qtv)/((1-qv)(1-tv))\\[1mm]
 (2,3)&3<s(2)&(1-v)/(1-qv)\\[1mm]
 (2,4)&4=s(2)&(1-v)/(1-tv)\\[1mm]
 (3,4)&s(3)\text{ undefined}&(1-v)/(1-qv)
\end{array}
\]
\end{example}
For an integer tuple
 {$\gamma=(\gamma_1,\ldots,\gamma_r)$}, define
{ 
\[\mathbf z^\gamma=z_1^{\gamma_1}\cdots z_r^{\gamma_r}\]
}
and the linear {evaluation maps}
\[
  {\operatorname{ev}_e(z_1^{\gamma_1}\cdots z_r^{\gamma_r})
 =e_{\gamma_1}(X)\cdots e_{\gamma_r}(X)},
 \qquad e_0=1,\quad e_d=0\quad(d<0)
\]
\[
 {\operatorname{ev}_h(z_1^{\gamma_1}\cdots z_r^{\gamma_r})
 =h_{\gamma_1}(X)\cdots h_{\gamma_r}(X)},
 \qquad h_0=1,\quad h_d=0\quad(d<0)
\]
 where $X$ is the alphabet of the symmetric functions. Let $\omega$ be the symmetric-function involution exchanging $h_n$
and $e_n$, thus, \[
 (\omega\circ  \operatorname{ev}_h)(\mathbf z^\gamma)
 = {\omega(h_{\gamma_1}\cdots h_{\gamma_r})}
 = {e_{\gamma_1}\cdots e_{\gamma_r}}
 =\operatorname{ev}_e(\mathbf z^\gamma).
\]

We use the Macdonald parameters $q,t$ of \cite{BHMPS}
throughout. 
The unsymmetrized series $\widehat{\mathbf H}_\beta(\mathbf z;q,t)$
appearing in \cite[Theorem 6.1.1]{BHMPS} is
\[
 \widehat{\mathbf H}_\beta(\mathbf z;q,t)
 =\frac{
  \displaystyle\prod_{(i,j)\in R_\beta\setminus\widehat R_\beta}
       (1-q^{a_i+1}t^{-\ell_i}z_i/z_j)
  \prod_{(i,j)\in\widehat R_\beta}(1-qtz_i/z_j)}
 {\displaystyle\prod_{(i,j)\in R_+}(1-qz_i/z_j)
  \prod_{(i,j)\in R_\beta}(1-tz_i/z_j)}.
\]
 {In this subsection, calculations take place over $\mathbb C(q,t)$,
with each denominator expanded geometrically in $z_i/z_j$. The evaluation
maps extend to these series termwise. Let $\mu\vdash n$, write
$\mu^{\mathsf t}$ for its conjugate partition, and let $\beta$ be any
rearrangement of $\mu^{\mathsf t}$.}
The \emph{modified Macdonald polynomial} can be expressed by  \cite[Equation~(57)]{BHMPS} as 
\begin{equation}\label{eq:modMac}
 \widetilde H_\mu(X;q,t)
 =\omega \circ \operatorname{ev}_h \left(
 z_1\cdots z_{ {n}}\prod_{i<j}(1-z_i/z_j)
                 \widehat{\mathbf H}_\beta(\mathbf z;q,t)\right).
\end{equation}

\begin{theorem}
\label{thm:BHMPSseries}
Let $\mu\vdash { {n}}$ and let $\beta$ be any rearrangement of
$ {\mu^{\mathsf t}}$. Let $\mathbf C^\beta(q,t)$ be the collection
in Definition~\ref{def:Macseriess}. 
 Then
\begin{equation}\label{eq:BHMPSrootform}
 \begin{aligned}
 \widetilde H_\mu(X;q,t)
=\Ue_{(1^{ {n}})}({\mathbf C^\beta(q,t)};e).
 \end{aligned}
\end{equation}
\end{theorem}
\begin{proof}

Since $R_\beta$ is the disjoint union of
$R_\beta\setminus\widehat R_\beta$ and $\widehat R_\beta$,
the displayed formula for $\widehat{\mathbf H}_\beta$ gives
{ 
\begin{align*}
 &\prod_{1\leq i<j\leq n}(1-z_i/z_j)\widehat{\mathbf H}_\beta\\
 &\quad=\prod_{(i,j)\in R_+}\frac{1-z_i/z_j}{1-qz_i/z_j}
 \cdot\prod_{(i,j)\in R_\beta\setminus\widehat R_\beta}
       \frac{1-q^{a_i+1}t^{-\ell_i}z_i/z_j}{1-tz_i/z_j}\\
 &\qquad{}\cdot\prod_{(i,j)\in\widehat R_\beta}
       \frac{1-qtz_i/z_j}{1-tz_i/z_j}\\
 &\quad=\prod_{1\leq i<j\leq n}C^\beta_{ij}(z_i/z_j;q,t).
\end{align*}
}
 {The first product contributes to every pair. The second and third products contribute precisely when $j=s(i)$ and $j>s(i)$, respectively, which are the other two cases of \eqref{eq:Macseriess}.}
Substituting this factorization into Equation \eqref{eq:modMac} proves
\[
 \widetilde H_\mu(X;q,t)
 =\operatorname{ev}_e\left(
 z_1\cdots z_{ {n}}\prod_{i<j}C^\beta_{ij}(z_i/z_j;q,t)\right).
\]

Let
$C^\beta_{ij}(v;q,t)=\sum_{k\geq0}c_{ij,k}v^k$ and choose
nonnegative integers $k_{ij}$ for all pairs. Then
\[
 z_1\cdots z_{ {n}}\prod_{i<j}(z_i/z_j)^{k_{ij}}
 =\prod_{i=1}^{ {n}} z_i^{\gamma_i}\]
 where
 \[
 \gamma_i=1+\sum_{j>i}k_{ij}-\sum_{j<i}k_{ji}.
\]
By definition, the same tuple $\gamma$ is obtained by applying
$\prod_{i<j}R_{ij}^{k_{ij}}$ to $(1^{ {n}})$. Therefore
\begin{align*}
 &\operatorname{ev}_e\left(
 z_1\cdots z_{ {n}}\prod_{i<j}C^\beta_{ij}(z_i/z_j;q,t)\right)\\
 &\quad=\sum_{\substack{k_{ij}\geq0\\\gamma_i\geq0\ (1\leq i\leq { {n}})}}
        \left(\prod_{i<j}c_{ij,k_{ij}}\right)
                      e_{\gamma_1}\cdots e_{\gamma_{ {n}}}\\
 &\quad=\prod_{i<j}C^\beta_{ij}(R_{ij};q,t)e_{(1^{ {n}})}
       =\Ue_{(1^{ {n}})}({\mathbf C^\beta(q,t)};e).
\end{align*}
\end{proof}

 {For an arbitrary root $\gelement$, let
$\mathbf{g}(u)=\gelement(-u)^{-1}$ and define the \emph{associated modified Macdonald lift by $\mathbf{g}$}}
{ 
\[
\widetilde{\mathsf M}_{\mu}(\gelement;q,t)
=\Ue_{(1^{|\mu|})}({\mathbf C^{\mu^{\mathsf t}}(q,t)};\mathbf{g}).
\]
}
 {The inverse signed series $\mathbf{g}$ is group-like, and this expression
belongs to $A_{\gelement}\otimes_{\mathbb C}\mathbb C(q,t)$.
The column order is fixed by $\mu^{\mathsf t}$; for noncommuting root
components, different rearrangements may give different lifts, and
negative powers of $t$ need not cancel. For the complete symmetric-function
root $h$, one has $\mathbf{g}=e$, so Theorem~\ref{thm:BHMPSseries} recovers
$\widetilde H_\mu(X;q,t)$.}

\subsection{Known bases obtained by choosing the raising series and root}
The following table records standard specializations.  
Here $P$ is the peak algebra generated by the components of
$Q(u)=H(-u)^{-1}H(u)$, and $\PC$ is the set of peak compositions $\alpha$,
$\alpha_i>1$ for $i<\ell(\alpha)$. The commutative Schur $Q$ root
is $q(u)=h(-u)^{-1}h(u)$.

\begin{center}
\small
\begin{tabular}{llll}
Algebra & Root & Raising series $C(v)$ & Result \\\hline
$\Sym$ & $h$ & $1$ & $h_\lambda$ \\
$\Sym$ & $e$ & $1$ & $e_\lambda$ \\
$\Sym$ & $h$ & $1-v$ & $s_\lambda$ \\
$\Sym$ & $e$ & $1-v$ & $s_{\lambda^{\mathsf t}}$ \\
$\Sym[q]$ & $h$ & $(1-v)/(1-qv)$ & $Q'_\lambda(X;q)$ \\
$\mathbb C[p_1,p_3,\ldots]$ & $q$ & $(1-v)/(1+v)$ & $Q_\lambda$ \\
$\NSym$ & $H$ & $1$ & $H_\alpha$ \\
$\NSym$ & $E$ & $1$ & $E_\alpha$ \\
$\NSym$ & $H$ & $1-v$ & $S_{(\alpha,H)}$ \\
$\NSym[q]$ & $H$ & $(1-v)/(1-qv)$ & $\widetilde S_{(\alpha,H)}(q)$ \\
$P$ & $Q$ & $(1-v)/(1+v)$ & $QS_{(\alpha,Q)}$, $\alpha\in\PC$
\end{tabular}
\end{center}

The last row and its commutative image are the raising formulas of
\cite[Proposition~3.6]{JL},
\[
 QS_{(\alpha,Q)}=\Ue_\alpha({(1-v)/(1+v)};Q).
\]

{ 
\section{Extensions to  symmetric functions in noncommuting variables}\label{sec:NCSym}

The study of { the Hopf algebra of} symmetric functions in noncommuting variables, {$\NCSym$,} goes back to
Wolf~\cite{Wolf} {who defined the elementary symmetric functions in noncommuting variables}. Rosas and Sagan~\cite{RS}
{defined} the monomial, power-sum, elementary,
and complete homogeneous bases indexed by set partitions, including
their transition formulas.  In \cite{ALvW}, {Aliniaeifard, Li and van Willigenburg} found a family of Schur functions in noncommuting variables {answering a question of Rosas and Sagan.} In \cite{CJ}, Cho and Jin used the action of permutations on monomials to construct Hall--Littlewood functions for $\NCSym$.

\subsection{The Hopf structure and the normalized group-like element}
 For a set partition $\pi$ of $[n]$, {where $\pi = B_1 / \cdots / B_r$ consists of disjoint subsets $B_1, \ldots ,B_r$ of  $[n]$ whose union in $[n]$, known as \emph{blocks}, let}
$\pi!=\prod_{{B_i}\in\pi}|{B_i}|!$. If $T\subseteq[n]$, let $\pi|_T$ be
the partition obtained by intersecting the blocks with $T$ and deleting
empty blocks. Standardization of $\pi|_T$, denoted by $\st(\pi|_T)$, is obtained by replacing the elements of $T$ increasingly
by $1,\ldots,|T|$. The \emph{concatenation} of $\pi\vdash [n]$ and $\tau\vdash [m]$, denoted by $\pi|\tau$, is the set partition whose blocks are the blocks of  $\pi$ and the blocks of $\tau$ shifted by $n$. Let $\{\hNCSym_\pi\}$ be the complete homogeneous basis of $\NCSym$. 
The complete basis satisfies
\begin{align*}
 \hNCSym_\pi\hNCSym_\tau&=\hNCSym_{\pi\mid\tau},
                                                \\
 \Delta\hNCSym_\pi
 &=\sum_{T\subseteq[n]}
   \hNCSym_{\st(\pi|_T)}\otimes
   \hNCSym_{\st(\pi|_{T^c})}.                    
\end{align*}

Consider the normalized series
\[
 \hNCSym=\sum_{n\geq0}\frac{\hNCSym_{[n]}}{n!},\qquad
 \hNCSym_n=\frac{\hNCSym_{[n]}}{n!}.
\]
Let $\hNCSym_0=1$ and $\hNCSym_n=0$ for $n<0$.
A set partition is \emph{atomic} if it has no nontrivial shifted
concatenation factorization. Every set partition has a unique
factorization into atomic partitions. 

\begin{proposition}\label{prop:Nfree}
The root $\hNCSym$ is free and group-like. Its universal morphism is
the Hopf embedding
\[
 J:\NSym\hookrightarrow\NCSym,\qquad
 J(H_n)=\frac{\hNCSym_{[n]}}{n!},\qquad
 J(H_\alpha)=\frac{\hNCSym_{[\alpha]}}{\alpha!},
\]
where $[\alpha]=[\alpha_1]\mid\cdots\mid[\alpha_r]$ and
$\alpha!=\alpha_1!\cdots\alpha_r!$.
\end{proposition}
\begin{proof}
We have that
$\Delta\hNCSym_{[n]}=\sum_{i=0}^n\binom ni
\hNCSym_{[i]}\otimes\hNCSym_{[n-i]}$.
Dividing by $n!$ proves the group-like identity, and the rest follows.
\end{proof}
{\begin{remark}
The map that lets the variables commute sends
$\hNCSym_\pi$ to $\pi!h_{\lambda(\pi)}$, where $\lambda(\pi)$
is the decreasing rearrangement of the block sizes. Thus, $\widetilde{\mathsf M}_{\mu}(\hNCSym;q,t)$ for $\NCSym$
maps to the corresponding modified Macdonald polynomials.
\end{remark}}

\subsection{Extending the source families}
Recall from \eqref{eq:seriesdefinition} that
\[
 \Ue_\alpha(C;\hNCSym)=\prod_{i<j}C(R_{ij})\hNCSym_\alpha.
\]
Factor $\pi$ into atomic partitions. In the normalized product
$\hNCSym_\pi/\pi!$, replace each maximal run of one-block atomic
factors $\hNCSym_{a_1}\cdots\hNCSym_{a_r}$ by
$\Ue_{(a_1,\ldots,a_r)}(C;\hNCSym)$. Keep all other normalized
atomic factors in their original order, and denote the resulting
element by $\Sp^C_\pi(\hNCSym)$.

\begin{example}
For $\pi=(14/23)\mid12\mid123\mid(145/23)\mid12\mid12\vdash[18]$,
\[
 \Sp_\pi^C(\hNCSym)
 =\frac{\hNCSym_{14/23}}{4}\,
   \Ue_{(2,3)}(C;\hNCSym)\,
   \frac{\hNCSym_{145/23}}{12}\,
   \Ue_{(2,2)}(C;\hNCSym).
\]
\end{example}

\begin{proposition}\label{prop:Nfullseries}
For $C(v)\in K[[v]]$ with constant term $1$, the elements
$\{\Sp^C_\pi(\hNCSym):\pi\vdash[n]\}$ form a basis of
$\NCSym_n\otimes_{\mathbb{C}}K$. For $C(v)=1-v$ and $C(v)=(1-v)/(1-qv)$, they extend
the Schur and Hall--Littlewood polynomials, respectively; in the
second case $K=\mathbb C[q]$.
\end{proposition}
\begin{proof}
Any set partition is a word on atomic set partitions. 
Order atomic partitions first by size, with arbitrary fixed tie-breaking. 
In a nonidentity raising monomial, the first affected factor increases in size and remains
{one block.} Thus, $$\Sp_\pi^C(\hNCSym)=\frac{\hNCSym_\pi}{\pi!}+\sum_{\tau> \pi} a_{\pi,\tau} {\frac{\hNCSym_\tau}{\tau!} .}$$
 Consequently, the transition matrix from the normalized complete
basis is therefore unitriangular in each degree.
\end{proof}

\subsection{Extending the raising operators}\label{subsec:Nambient}
For this construction, we use the basis
$\{\eNCSym_\pi/\pi!\}$ of $\NCSym$, where $\eNCSym_\pi$ denotes
the elementary symmetric function indexed by the set partition $\pi$. For the empty partition $\emptyset$, we let $\eNCSym_{\emptyset}=1$. 
We have that
\[
 \eNCSym_\pi\eNCSym_\tau=\eNCSym_{\pi\mid\tau},\qquad
 \eNCSym_n=\frac{\eNCSym_{[n]}}{n!},\qquad
 \eNCSym=\sum_{n\geq0}\eNCSym_n,
\]
with $\eNCSym_0=1$ and $\eNCSym_n=0$ for $n<0$.
We have that
$\Delta\eNCSym_{[n]}=\sum_{i=0}^n\binom ni
\eNCSym_{[i]}\otimes\eNCSym_{[n-i]}$ and  
show that $\eNCSym$ is a free root. We have the 
involution $\omega_{\NCSym}(\hNCSym_\pi)=\eNCSym_\pi$
 takes one free root to another.

For an atomic set partition $\rho\vdash[m]$, define $\rho\langle d\rangle$
for integers $d\geq0$ as follows. For $0\leq d\leq m$, restrict $\rho$
to $[d]$. For $d>m$, adjoin $m+1,\ldots,d$ to the block containing
$m$. In particular, $\rho\langle0\rangle=\emptyset$,
$\rho\langle m\rangle=\rho$, and $\rho\langle d\rangle$ is
atomic for every $d>m$.  For example,
$(125/34)\langle 7\rangle =12567/34$ and
$(125/3467)\langle 5\rangle=125/34$.

Write $\pi=\pi^{(1)}\mid\cdots\mid\pi^{(r)}$ for its atomic
factorization and let $\mathbf{n}(\pi)=(n_1,\ldots,n_r)$, where
$n_i=|\pi^{(i)}|$.  
For
nonnegative integers $k_{ij}$, set
\[
 d_i=n_i+\sum_{j>i}k_{ij}-\sum_{j<i}k_{ji},
\]
and define the action of a monomial in the raising operators by
\[
 \left(\prod_{i<j}R_{ij}^{k_{ij}}\right)
       \frac{\eNCSym_\pi}{\pi!}
 =\begin{cases}
 \displaystyle\prod_{i=1}^r
   \frac{\eNCSym_{\pi^{(i)}\langle d_i\rangle}}
        {\pi^{(i)}\langle d_i\rangle!},&d_i\geq0\text{ for every }i,\\[6pt]
 0,&\text{otherwise}.
 \end{cases}
\]
The product is taken from left to right, with an empty partition
contributing $1$. 
For each length $r$, choose scalar series
$\mathbf C=(C_{ij})_{1\leq i<j\leq r}$ with
$C_{ij}(v)=\sum_{k\geq0}c_{ij,k}v^k\in K[[v]]$ and $c_{ij,0}=1$. Define
\begin{equation}\label{eq:Npairseries}
 \Ue_\pi({\mathbf C};\eNCSym)
 =\left(\prod_{1\leq i<j\leq r}C_{ij}(R_{ij})\right)
       \frac{\eNCSym_\pi}{\pi!}.
\end{equation}

\begin{theorem}\label{prop:mainW} 
For each length $r$, let
$\mathbf C=(C_{ij})$ with
$C_{ij}(v)=\sum_{k\geq0}c_{ij,k}v^k$ and $c_{ij,0}=1$. The elements
$\{\Ue_\pi({\mathbf C};\eNCSym):\pi\vdash[n]\}$ form a basis
of $\NCSym_n\otimes_{\mathbb{C}}K$. Taking every $C_{ij}(v)$ equal to $1-v$ or to
$(1-v)/(1-qv)$ extends the Schur or modified
Hall--Littlewood lifts, respectively. 
\end{theorem}

\begin{proof}
Write $\pi=\pi^{(1)}\mid\cdots\mid\pi^{(r)}$ for its atomic
factorization and let $\mathbf n=(n_1,\ldots,n_r)$, where
$n_i=|\pi^{(i)}|$.
A monomial $\prod_{i<j}R_{ij}^{k_{ij}}$ gives the final degrees
$d_i$ above. If all $d_i$ are nonnegative, then
\[
 \sum_{i<j}(j-i)k_{ij}
 =\sum_i i n_i-\sum_i i d_i\leq\sum_i i n_i.
\]
Only finitely many raising-operator monomials survive.
The identity monomial contributes $\eNCSym_\pi/\pi!$. Use the lexicographic order on atomic factor words from the proof of Proposition \ref{prop:mainW}. 
Action of a nonidentity monomial of raising operators on $\eNCSym_\pi/\pi!$, changes the sizes of the atomic partitions after factorization, so that the first changed atomic factor increases in size. Therefore,
\[
\Ue_\pi(\mathbf{C};\eNCSym)=\frac{\eNCSym_\pi}{\pi!}+\sum_{\tau>\pi} c_{\pi,\tau}\frac{\eNCSym_\tau}{\tau!}.
\]
Thus, the transition matrix from the normalized elementary
basis is therefore unitriangular in each degree.
\end{proof}
For a fixed $\mu\vdash n$, take $K=\mathbb C(q,t)$ and $\mathbf C=\mathbf C^{\mu^{\mathsf t}}(q,t)$. The basis element indexed by the  set partition $1|2|\cdots|n$ is
\[
 \Ue_{1|2|\cdots|n}(\mathbf C;\eNCSym)
 =\Ue_{(1^n)}(\mathbf C^{\mu^{\mathsf t}}(q,t);\eNCSym),
\]
which maps to $\widetilde H_\mu(X;q,t)$ after commuting the variables.

\subsection{The permutation action and Schur functions}
For $\delta\in\mathfrak S_n$, define its action on words of length $n$ by
\begin{equation}\label{eq:Naction}
 \delta\circ(x_{i_1}\cdots x_{i_n})
 =x_{i_{\delta^{-1}(1)}}\cdots x_{i_{\delta^{-1}(n)}}
\end{equation}
{and extend} linearly. We have that
$\delta\circ\hNCSym_\pi=\mathbf \hNCSym_{\delta(\pi)}$, so this action
preserves $\NCSym_n$. To specify the permutation $\delta_\pi$ of
\cite[Section~4]{ALvW}, order the blocks of $\pi$ by decreasing size,
breaking ties by increasing least element, and list each block increasingly.
Their concatenation is the one-line notation of $\delta_\pi$.
Writing $\lambda=\lambda(\pi)$, the source Schur function and the
standard Schur function of \cite{ALvW} are, respectively,
\[
 s_{[\lambda]}=\Ue_\lambda({1-v};\hNCSym),\qquad
 s_\pi=\delta_\pi\circ\Ue_\lambda({1-v};\hNCSym).
\]
}
It would be interesting to find all 
     suitable permutations $\sigma_\pi$ for each partition $\pi$ such that $\{\sigma_\pi \circ \widetilde{\mathsf M}_{\lambda(\pi)}(\hNCSym;q,t) : \pi\vdash [n]\}$ is a basis for $\NCSym_n$.

\section{Extensions to the Malvenuto--Reutenauer Hopf algebra of permutations}\label{sec:FQSym}

The Malvenuto--Reutenauer Hopf algebra, {$\FQSym$,}  was introduced in \cite{MR}, and its basis is indexed by permutations. 

\subsection{The Hopf structure and the normalized group-like element}
For a permutation $\sigma\in\mathfrak S_n$, write
$\sigma=\sigma_1\cdots\sigma_n$ in one-line notation. Let $|\sigma|=n$. Let $X_n^*$ denote the set of words of length $n$ with letters in the positive integers.
The standardization of a word $w={i_1}{i_2}{\cdots} {i_n}\in X_n^*$,  denoted by $\st(w)$, is the permutation $\sigma_1\sigma_2\cdots \sigma_n$ where 
$$\sigma_t=\#\{j:i_j<i_t\} + \{j\leq t:i_j=i_t\}.$$ For $\sigma\in \mathfrak{S}_n$,  
define  
$$\mathbf{G}_\sigma=\sum_{i_1i_2{\cdots} i_n\in X_n^* \atop \st(i_1i_2{\cdots} i_n)=\sigma} x_{i_1}x_{i_2}\cdots x_{i_n}.$$
Then  $\{\mathbf G_\sigma\}$ denotes the
standard basis {of $\FQSym$} in the convention
$\mathbf G_\sigma=\mathbf F_{\sigma^{-1}}$ of~\cite{DHT}.
For $A\subseteq[n]$, let $\sigma|_A$ be the subword consisting of
letters in $A$. Let $\sigma\in\mathfrak{S}_n$ and $\tau\in\mathfrak S_m$.  The product and coproduct are
\begin{align*}
 \mathbf G_\sigma\mathbf G_\tau
 &=\sum_{\substack{\gamma\in\mathfrak S_{n+m}\\
        \st(\gamma_1\cdots\gamma_n)=\sigma\\
        \st(\gamma_{n+1}\cdots\gamma_{n+m})=\tau}}
        \mathbf G_\gamma,
                                                        \\
 \Delta\mathbf G_\sigma
 &=\sum_{i=0}^n
    \mathbf G_{\st(\sigma|_{[i]})}\otimes
    \mathbf G_{\st(\sigma|_{\{i+1,\ldots,n\}})}. 
\end{align*}

The \emph{right--shifted concatenation} and \emph{left--shifted concatenation} of
$\sigma\in\mathfrak S_n$ and $\tau\in\mathfrak S_m$ are
\[
 \sigma  \curvearrowright \tau=\sigma_1\cdots\sigma_n
                       (\tau_1+n)\cdots(\tau_m+n)\qquad \text{and} \qquad
 \sigma \curvearrowleft   \tau=(\sigma_1+m)\cdots(\sigma_n+m)
                       \tau_1\cdots\tau_m,
\]
respectively. 
A nonempty permutation is \emph{right-atomic}, respectively
\emph{left-atomic}, if it has no nontrivial right--shifted, respectively
left--shifted, concatenation factorization. Every permutation has a unique
factorization of each kind. If
$\sigma=\sigma^{(1)}\curvearrowleft\cdots\curvearrowleft\sigma^{(r)}$
is its left-atomic factorization, define
\[
 \mathbf{H}_\sigma
 =\mathbf G_{\sigma^{(1)}}\cdots\mathbf G_{\sigma^{(r)}}.
\]
Then $\mathbf{H}_\sigma\mathbf{H}_\tau=\mathbf{H}_{\sigma\curvearrowleft\tau}$.
Consider that the set
$\{\mathbf{H}_\sigma:\sigma\in\mathfrak S_n\}$ is a basis of
$\FQSym_n$.

\subsection{Extending the source families}
Let $\iota_n=12\cdots n$. Consider that the series
\[
 \mathbf{H}=\sum_{n\geq0}\mathbf{H}_n,
\]
where 
 $\mathbf{H}_n=\mathbf{H}_{\iota_n}=\mathbf G_{\iota_n}$ is a  free root of ${\FQSym}$. 
Let $\mathbf{H}_0=1$ and $\mathbf{H}_n=0$ for $n<0$.

Recall that for $C(v)\in K[[v]]$, 
\[
 \Ue_\alpha(C;\mathbf H)
 =\prod_{1\leq i<j\leq r}C(R_{ij})\mathbf H_{\alpha_1}\cdots\mathbf H_{\alpha_r}.
\]
Factor $\sigma$ into left-atomic permutations. In the product
$\mathbf H_\sigma$, replace each maximal run of increasing left--atomic
factors $\mathbf H_{a_1}\cdots\mathbf H_{a_r}$ by
$\Ue_{(a_1,\ldots,a_r)}(C;\mathbf H)$. Keep all other left--atomic factors
in their original order, and denote the resulting element by
$\Sp^C_\sigma(\mathbf H)$.

\begin{example}
For $\sigma=132\curvearrowleft12\curvearrowleft123\curvearrowleft213\curvearrowleft12\curvearrowleft12$,
\[
 \Sp^C_\sigma(\mathbf H)
 =\mathbf G_{132}\,
   \Ue_{(2,3)}(C;\mathbf H)\,
   \mathbf G_{213}\,
   \Ue_{(2,2)}(C;\mathbf H).
\]
\end{example}

The proof of the following proposition is analogous to that of  Proposition \ref{prop:Nfullseries}. 
\begin{proposition}\label{prop:Ffullkernel}
For $C(v)\in K[[v]]$ with constant term $1$, the elements
$\{\Sp^C_\sigma(\mathbf H):\sigma\in\mathfrak S_n\}$ form a basis
of $\FQSym_n\otimes_{\mathbb C}K$. For $C(v)=1-v$ and
$C(v)=(1-v)/(1-qv)$, they extend the Schur and modified
Hall--Littlewood lifts for $\mathbf H$, respectively; in the second
case $K=\mathbb C[q]$.
\end{proposition}

\subsection{Extending the raising operators}
If $\sigma=\sigma^{(1)}\curvearrowright\cdots\curvearrowright\sigma^{(r)}$ is its
right--atomic factorization, define
\[
 \mathbf E_\sigma
 =\mathbf G_{\sigma^{(1)}}\cdots\mathbf G_{\sigma^{(r)}}.
\]
Let $\omega_n=n {\cdots} 21$. Consider that 
\[
 \mathbf E=\sum_{n\geq0}\mathbf E_n,
\]
where $\mathbf E_n=\mathbf E_{\omega_n}=\mathbf G_{\omega_n},$ is a free root of $\FQSym$. 
Let  $\mathbf E_0=1$ and $\mathbf E_n=0$ for $n<0$.
  Moreover, for $n>0$ the coefficient of every $\mathbf G_\sigma$ in
$\sum_{i=0}^n(-1)^i\mathbf G_{\iota_i}\mathbf G_{\omega_{n-i}}$
is zero. Therefore, 
\[
 \mathbf E(u)=\mathbf H(-u)^{-1}.
\]

For a right--atomic permutation $\rho\in\mathfrak S_m$, define
$\rho\langle d\rangle$ for integers $d\geq0$ as follows. For
$0\leq d\leq m$, $\rho\langle d\rangle$ is obtained by standardizing the first $d$ letters of $\rho$.
For $d>m$, let
\[
 \rho\langle d\rangle=d(d-1)\cdots(m+1)\rho_1\cdots\rho_m.
\]
In particular, $\rho\langle0\rangle=\emptyset$,
$\rho\langle m\rangle=\rho$, and $\rho\langle d\rangle$ is
right--atomic for every $d>m$, since its first letter is $d$.
For example, $231\langle5\rangle=54231$ and
$231\langle2\rangle=12$. Also
$\omega_m\langle d\rangle=\omega_d$.

When $\sigma=\sigma^{(1)}\curvearrowright\cdots\curvearrowright\sigma^{(r)}$, let
$\mathbf n(\sigma)=(n_1,\ldots,n_r)$, where
$n_i=|\sigma^{(i)}|$. 
For nonnegative integers $k_{ij}$, set
\[
 d_i=n_i+\sum_{j>i}k_{ij}-\sum_{j<i}k_{ji},
\]
and define the evaluation of a raising monomial by
\[
 \left(\prod_{i<j}R_{ij}^{k_{ij}}\right)\mathbf E_\sigma
 =\begin{cases}
  \displaystyle\prod_{i=1}^r
       \mathbf E_{\sigma^{(i)}\langle d_i\rangle},
       &d_i\geq0\text{ for every }i,\\[4pt]
  0,&\text{otherwise}.
 \end{cases}
\]
The product is taken from left to right, with an empty permutation
contributing $1$. The original slots are retained until the whole
monomial is evaluated.
For each length $r$, choose scalar series
$\mathbf C=(C_{ij})_{1\leq i<j\leq r}$ with
$C_{ij}(v)=\sum_{k\geq0}c_{ij,k}v^k\in K[[v]]$ and $c_{ij,0}=1$.
Define
\begin{equation*}\label{eq:Fpairkernel}
 \Ue_\sigma(\mathbf C;\mathbf E)
 =\left(\prod_{1\leq i<j\leq r}C_{ij}(R_{ij})\right)\mathbf E_\sigma.
\end{equation*}

The proof of the following proposition is analogous to that of  Theorem \ref{prop:mainW}. 
\begin{theorem}\label{thm:Fambient}
For each length $r$, choose scalar series $\mathbf C=(C_{ij})$
with constant terms $1$. The elements
$\{\Ue_\sigma(\mathbf C;\mathbf E):\sigma\in\mathfrak S_n\}$
form a basis of $\FQSym_n\otimes_{\mathbb C}K$. Taking every
$C_{ij}(v)$ equal to $1-v$ or to $(1-v)/(1-qv)$ extends the Schur
or modified Hall--Littlewood lifts for the root
$\mathbf E$, respectively.
\end{theorem}

For a fixed $\mu\vdash n$, take $K=\mathbb C(q,t)$ and choose
$\mathbf C=\mathbf C^{\mu^{\mathsf t}}(q,t)$ for length $n$,
with any constant-term-one choices for the other lengths.
The basis element indexed by $\iota_n=1\curvearrowright\cdots\curvearrowright1$ is
\[
 \Ue_{\iota_n}(\mathbf C;\mathbf E)
 =\Ue_{(1^n)}(\mathbf C^{\mu^{\mathbf t}}(q,t);\mathbf E)
 =\widetilde{\mathsf M}_\mu(\mathbf H;q,t).
\]
Its commutative image is $\widetilde H_\mu(X;q,t)$.

 \begin{remark}
     Similar constructions in this section can be carried out in the Hopf algebra of quasisymmetric functions in noncommuting variables \cite{Hivert}. 
 \end{remark}

\section*{Acknowledgments}
The authors made minimal use of OpenAI Codex, mainly for editorial assistance.
All mathematical ideas presented in this paper are the authors' own.

\enlargethispage{3\baselineskip}


\begin{thebibliography}{99}
\bibitem{ABS} M. Aguiar, N. Bergeron, and F. Sottile,
\emph{Combinatorial Hopf algebras and generalized Dehn--Sommerville relations},
Compos. Math. \textbf{142} (2006), no.~1, 1--30.

\bibitem{ALvW} F. Aliniaeifard, S. X. Li, and S. van Willigenburg,
\emph{Schur functions in noncommuting variables},
Adv. Math. \textbf{406} (2022), article 108536.

\bibitem{BB} M. M. Bayer and L. J. Billera,  \emph{Generalized Dehn-Sommerville relations for polytopes, spheres and Eulerian partially ordered sets}, Invent. Math. \textbf{79} (1985), no.~1, 143--157. 

\bibitem{BBS} C. Berg, N. Bergeron, F. Saliola, L. Serrano, and M. Zabrocki,
\emph{A lift of the Schur and Hall--Littlewood bases to non-commutative symmetric functions},
Canad. J. Math. \textbf{66} (2014), no.~3, 525--565.


\bibitem{BGHT} F. Bergeron, A. M. Garsia, M. Haiman, and G. Tesler,
\emph{Identities and positivity conjectures for some remarkable operators
in the theory of symmetric functions},
Methods Appl. Anal. \textbf{6} (1999), no.~3, 363--420.



\bibitem{BHMPS} J. Blasiak, M. Haiman, J. Morse, A. Pun, and G. H. Seelinger,
\emph{A raising operator formula for Macdonald polynomials},
Forum Math. Sigma \textbf{13} (2025), 13:e47 1--18.
\href{https://doi.org/10.1017/fms.2025.8}{doi:10.1017/fms.2025.8}.

{\bibitem{CM} E. Carlsson and A. Mellit,
\emph{A proof of the shuffle conjecture},
J. Amer. Math. Soc. \textbf{31} (2017), no.~3, 661--697.}


\bibitem{CJ} S. Cho and E. Y. Jin,
\emph{Hall--Littlewood functions in noncommuting variables},
\href{https://arxiv.org/abs/2606.21041}{arXiv:2606.21041}, 2026.

\bibitem{DHT} G. Duchamp, F. Hivert, and J.-Y. Thibon,
\emph{Noncommutative symmetric functions VI: Free quasi-symmetric functions and related algebras},
Internat. J. Algebra Comput. \textbf{12} (2002), no.~5, 671--717.


\bibitem{GarsiaMilne} A. M. Garsia,
\emph{Orthogonality of Milne's polynomials and raising operators},
Discrete Math. \textbf{99} (1992), nos.~1--3, 247--264.


\bibitem{GKLRT} I. M. Gelfand, D. Krob, A. Lascoux, B. Leclerc,
V. S. Retakh, and J.-Y. Thibon,
\emph{Noncommutative symmetric functions},
Adv. Math. \textbf{112} (1995), no.~2, 218--348.

\bibitem{GR} D. Grinberg and V. Reiner,
\emph{Hopf Algebras in Combinatorics},
arXiv:1409.8356 [math.CO], 2020.

{ 
\bibitem{HMZ} J. Haglund, J. Morse, and M. Zabrocki,
\emph{A compositional shuffle conjecture specifying touch points of the Dyck path},
Canad. J. Math. \textbf{64} (2012), no.~4, 822--844.
}

{\bibitem{Haiman} M. Haiman, 
\emph{Hilbert schemes, polygraphs, and the Macdonald positivity conjecture},
J. Amer. Math. Soc. \textbf{14} (2001), no.~4, 941--1006.
}

\bibitem{Hivert} F. Hivert, \emph{Combinatoire des fonctions quasi-sym{\'e}triques}, PhD thesis Universit{\'e} de Marne-la-Vall{\'e}e, 1999. 1999MARN0051.

\bibitem{Jing} N. H. Jing,
\emph{Vertex operators and Hall--Littlewood symmetric functions},
Adv. Math. \textbf{87} (1991), no.~2, 226--248.

\bibitem{JL} N. Jing and Y. Li,
\emph{A lift of Schur's $Q$-functions to the peak algebra},
J. Combin. Theory Ser. A \textbf{135} (2015), 268--290.

{ 
\bibitem{KN} A. N. Kirillov and M. Noumi,
\emph{$q$-Difference raising operators for Macdonald polynomials and the
integrality of transition coefficients},
in \emph{Algebraic Methods and $q$-Special Functions},
CRM Proc. Lecture Notes \textbf{22}, Amer. Math. Soc., Providence, RI,
1999, 227--243; arXiv:q-alg/9605005.
}

{ 
\bibitem{LV} L. Lapointe and L. Vinet,
\emph{Rodrigues formulas for the Macdonald polynomials},
Adv. Math. \textbf{130} (1997), no.~2, 261--279.
}

\bibitem{Mac} I. G. Macdonald,
\emph{Symmetric Functions and Hall Polynomials}, second edition,
Oxford University Press, Oxford, 1995.

\bibitem{MR} C. Malvenuto and C. Reutenauer,
\emph{Duality between quasi-symmetric functions and the Solomon descent algebra},
J. Algebra \textbf{177} (1995), no.~3, 967--982.


\bibitem{RS} M. H. Rosas and B. E. Sagan,
\emph{Symmetric functions in noncommuting variables},
Trans. Amer. Math. Soc. \textbf{358} (2006), no.~1, 215--232.



\bibitem{SZ} M. Shimozono and M. Zabrocki,
\emph{Hall--Littlewood vertex operators and generalized Kostka polynomials},
Adv. Math. \textbf{158} (2001), no.~1, 66--85.



\bibitem{Wolf} M. C. Wolf,
\emph{Symmetric functions of non-commutative elements},
Duke Math. J. \textbf{2} (1936), no.~4, 626--637.

\bibitem{Zel} A. V. Zelevinsky, \emph{Representations of finite classical groups. A Hopf algebra approach.} Lecture Notes in Mathematics \textbf{869}, Springer-Verlag, Berlin-New York, 1981.

\bibitem{Zhao} W. Zhao,
\emph{Noncommutative symmetric systems over associative algebras},
J. Pure Appl. Algebra \textbf{210} (2007), no.~2, 363--382.

\end{thebibliography}
\end{document}